\documentclass[a4 paper, 12pt, reqno]{amsart}
\usepackage{latexsym}
\usepackage[T1]{fontenc}
\usepackage[english]{babel}
\usepackage{amssymb,amsmath,amsthm,amsfonts}
\usepackage[latin1]{inputenc}

\usepackage{url}
\usepackage{a4wide}
\usepackage{enumerate}
\usepackage{changebar}
\usepackage{booktabs}
\usepackage{paralist}
\usepackage{comment}
\usepackage{mathrsfs}

\usepackage{graphicx}
\usepackage{color}
\usepackage{cite}

\theoremstyle{plain}
\newtheorem{theorem}{Theorem}[section]
\newtheorem{prop}[theorem]{Proposition}
\newtheorem{lemma}[theorem]{Lemma}
\newtheorem{corollary}[theorem]{Corollary}
\theoremstyle{definition}

\newtheorem{remark}[theorem]{Remark}

\def\R{{\mathbb R}}

\def\N{{\mathbb N}}

\def\P{{\mathbb P}}
\def\D{{\mathcal{D}}}
\def\eps{\varepsilon}
\def\OU{\mathcal{L}}

\def\L{\mathscr{L}}

\def\B{\mathrm{B}}
\def\BB{\mathbb{B}}
\def\Id{\mathrm{Id}}
\def\d{\mathrm{d}}

\def\p{\mathrm p}
\def\div{\mathrm{div\,}}

\def\q{\mathrm{q}}

\numberwithin{equation}{section}

\title[The Oseen-Navier-Stokes flow]{\textbf{The Oseen-Navier-Stokes flow in the exterior of a rotating obstacle: The non-autonomous case}}
\author{Tobias Hansel}
\address{ Technische Universit\"at Darmstadt\\Department of Mathematics\\ 64289 Darmstadt,
Germany} \email{hansel@mathematik.tu-darmstadt.de}
\author{Abdelaziz Rhandi}
\address{Dipartimento di Ingegneria dell'Informazione e Matematica Applicata \\ Universit\`a degli Studi di Salerno\\ Via Ponte Don Melillo, \\ 84084 Fisciano
(Sa)\\Italy}
\email{rhandi@diima.unisa.it}

\begin{document}

%
%

\maketitle 

%
%
%
\begin{abstract}
Consider the Navier-Stokes flow past a rotating obstacle with
a general time-dependent angular velocity and a time-dependent
outflow condition at infinity -- sometimes called an Oseen condition. By a suitable change
of coordinates the problem is transformed to an non-autonomous problem with unbounded drift terms
on a fixed exterior domain $\Omega\subset \R^d$. It is shown that the solution to the linearized problem is governed by a strongly
continuous evolution system $\{T_\Omega(t,s)\}_{t\geq s\geq0}$ on $L^p_\sigma(\Omega)$ for $1<p<\infty$. Moreover, $L^p$-$L^q$
smoothing properties and gradient estimates of $T_\Omega(t,s)$, $0\leq s \leq t$, are obtained. These results are the key ingredients to
show local in time existence of mild solutions to the full nonlinear problem for $p\geq d$ and initial value in $L^p_\sigma(\Omega)$.
\end{abstract}

\section{Introduction}
In this paper we consider the flow of an incompressible, viscous fluid
in the exterior of a rotating obstacle subject to an additional time-dependent outflow condition at
infinity. Here the angular velocity of the obstacle and the outflow condition at infinity may depend on time
and also the axis of rotation may change.
The equations describing this problem are the Navier-Stokes
equations in a time-dependent exterior domain with a prescribed velocity field at infinity.

After rewriting the problem on a fixed exterior domain $\Omega \subset \R^d$, we obtain an non-autonomous system of equations involving a family of time-dependent operators of the form
\begin{equation}\label{eq:Op}
A(t)u= \P_\Omega\left( \Delta u + (M(t)x + c(t))\cdot \nabla u - M(t) u\right), \qquad t\ge 0.
\end{equation}
where $\P_\Omega$ denotes the Helmholtz projection from $L^p(\Omega)^d$ onto the solenoidal space $L^p_\sigma(\Omega)$ and $M\in C^1([0,\infty);\R^{d\times d})$, $c\in C^1([0,\infty);\R^d)$.  The  main difficulty in dealing with these operators arises since the
term $M(t)x \cdot \nabla$ has unbounded coefficients in the exterior domain $\Omega$. In particular, the
lower order terms cannot be treated by classical perturbation theory for the
Stokes operator. In the autonomous case
$M(t)\equiv M$ and $c(t)\equiv 0$ such an operator was first considered by Hishida \cite{Hishida:1999} and then later by Geissert, Heck, Hieber \cite{Geissert/Heck/Hieber:2006a}.
It is the aim of this paper to extend their result to the non-autonomous case.

In the following let us briefly motivate our problem and let us show why it is interesting to study the non-autonomous case. For this purpose let $\mathcal O \subset \R^d$ be a compact obstacle with
smooth boundary and let $\Omega:=\R^d \setminus \mathcal O$ be the
exterior of the obstacle. Furthermore, let $m \in C^1([0,\infty); \R^{d\times d})$ be a
matrix-valued function that describes the velocity of the obstacle.
Then, the exterior of the rotated obstacle at time $t\ge 0$ is represented by $\Omega(t):=
Q(t)\Omega$ where $Q\in C^1([0,\infty),\R^{d\times d})$ solves the ordinary differential equation
\begin{equation}
\left\{\begin{array}{rcll}
\partial_t Q(t) &=& m(t)Q(t),& t\ge0,\\[0.2cm]
Q(0)&=&\mathrm{Id}.
\end{array}\right.
\end{equation}\normalsize
With a prescribed velocity field $v_\infty\in C^1([0,\infty);\R^d)$
at infinity, the equations for the fluid on the time-dependent
domain $\Omega(t)$ with no-slip boundary condition
take the form
\begin{align}\label{eq:NS_2}
v_t-\Delta v +v\cdot \nabla v+\nabla \q&=0&\quad\quad\quad\mbox{in $(0,\infty)\times \Omega(t)$,}\notag\\
\div v&=0&\quad\quad\quad\mbox{in $(0,\infty)\times \Omega(t)$,}\notag\\
v(t,y)&=m(t)y&\quad\quad\quad\mbox{ on $(0,\infty)\times \partial\Omega(t)$,}\\
\lim_{|y|\to \infty} v(t,y)&=v_\infty(t)&\quad\quad\quad\mbox{ \mbox{for} $t\in(0,\infty) $,}\notag\\
v(0,y)&=v_0(y)&\quad\quad\quad\mbox{in $\Omega$}\notag.
\end{align}
Here $v=v(y,t)$ and $\q=\q(y,t)$ are the unknown velocity field and the pressure of
the fluid, respectively.

%
As usual, it is reasonable to reduce \eqref{eq:NS_2} to a new problem on a fixed exterior domain by some suitable coordinate transformation. Since $m(\cdot)$ is the velocity of the rotated obstacle, it is natural to assume that $m(t)$ is skew
symmetric for all $t\ge0$. This implies that for all $t\ge 0$ the matrix $Q(t)$ is orthogonal. Thus, we can set\vspace{-0.2cm}
\begin{equation}
x=Q(t)^{\mathrm T}y, \quad u(t,x)=Q(t)^{\mathrm T}
(v(t,y)-v_\infty(t)), \quad \p(t,x)=\q(t,y).
\end{equation}
Then we obtain the following new system of equations on the reference domain $\Omega$:
\begin{align}\label{eq:NS_3}
	\left.
	\begin{array}{l}
u_t- \Delta u - M(t)x \cdot \nabla u + M(t)u  \\
\quad+ Q(t)^{\mathrm T}v_{\infty}(t)\cdot \nabla u
+Q(t)^{\mathrm T}\partial_t v_\infty(t)\\
\quad +u\cdot \nabla u
+\nabla \p
\end{array}\right\}&=0&\;\mbox{in   $(0,\infty)\times \Omega $,}\notag\\
\div u&=0&\quad\mbox{in $ (0,\infty)\times \Omega $,}\notag\\
u(t,x)&= M(t)x-Q(t)^{\mathrm T}v_\infty(t)&\;\mbox{on $(0,\infty)\times \partial\Omega $},\\
\lim_{|x|\to \infty} u(t,x)&= 0&\;\mbox{\mbox{for} $t\in(0,\infty) $,}\notag\\
u(0,x)&=u_0(x)&\;\mbox{in $\Omega$}.\notag
\end{align}
\normalsize
Here $M(t):=Q(t)^{\mathrm T}m(t)Q(t)$ is the transformed velocity of the obstacle. The coordinate transformation also ensures that the new velocity field $u$ vanishes at infinity, which is a natural condition in the $L^p$-setting.

Note that a problem of this type also arises in the analysis of a rotating body with a translational velocity $-v_\infty(t)$ by a similar coordinate transformation, see e.g. the explanations in \cite{Farwig:2005}.

%
Problem \eqref{eq:NS_2} was studied intensively for the special case of \emph{time-independent} matrices $M(t)\equiv M$ and
{\em without} an outflow condition, i.e.  $v_\infty\equiv0$. Hishida \cite{Hishida:1999} showed that the solution to the linearized problem is governed by a strongly continuous semigroup on $L^2$, which is however \emph{not analytic}. Moreover, he constructed local mild solutions
in $L^2$ by using the Fujita-Kato approach (cf. \cite{Fujita/Kato:1964}). Later this generation and existence result was extended to the general $L^p$-theory by Geissert, Heck,
Hieber \cite{Geissert/Heck/Hieber:2006a}. Hishida and Shibata \cite{Hishida/Shibata:2009} were even able to show global extistence for small data. The model problem in $\R^d$ was studied by Hieber and Sawada \cite{Hieber/
Sawada:2005} in the $L^p$-setting.

The case of \emph{time-dependent} angular velocities was considered by Hishida \cite{Hishida:2001} in the $L^2$-context, however he assumes that the axis of rotation is fixed and that the sign of his angular velocity does not change.

For the problem including an additional outflow condition at infinity,
there are only a few results.
The case, where $M(t)x=\omega(t) \times x$ and
$\omega:[0,\infty) \rightarrow \R^3$ is the angular velocity of the
obstacle and $v_\infty:[0,\infty)\to\R^3$ a time-dependent outflow velocity was considered by
Borchers \cite{Borcher:1992} in the framework of weak solutions. This work was somehow the starting point in the analysis of viscous fluid flow past a rotating obstacle. More recently,
Shibata \cite{Shibata:2008} studied
the special case where $M(t)\equiv M$, $v_\infty(t)\equiv v_\infty$ and
$Mv_\infty=0$. The additional condition $Mv_\infty=0$,  i.e.
$Q(t)^{\mathrm T}v_\infty= kv_\infty$ for $k\in \{+1,-1\}$, ensures that
\eqref{eq:NS_3} is still an autonomous equation.
The physical meaning of the additional condition is that the outflow
direction of the fluid is parallel to the axis of rotation of the
obstacle. The stationary problem of this latter situation was
analysed by Farwig \cite{Farwig:2005} for the whole space case $\R^3$.

In order to relax the assumption $Mv_\infty=0$, i.e. in order to allow a general  $v_\infty$, it is necessary to study a non-autonomous problem. Thus, in this context it is natural even to allow  time-dependent outflow velocities $v_\infty(\cdot)$ and time-dependent angular velocities $M(\cdot)$. The non-autonomous model problem of \eqref{eq:NS_3} in the case $\Omega=\R^d$ was recently studied by the first author \cite{Hansel:2009} and by Geissert and the first author \cite{Geissert/Hansel:2010}. Indeed, they were able to show that the family of operators in \eqref{eq:Op}, equipped with suitable domains, generate a strongly continuous evolution system on $L^p_\sigma(\R^d)$, $1<p<\infty$, which enjoys nice regularity properties. Their approach is based on an explicit solution formula for the linearized problem. By a version of Kato's iteration scheme (cf. \cite{Kato:1984,Giga:1986}) one obtains a (local) mild solution to the nonlinear problem on $\R^d$ for initial value $v_0 \in L^p_\sigma (\R^d)$, $d\leq p < \infty$. In this paper we use their results for the linearized problem to cover the physically more realistic situation of exterior domains by some cut-off techniques.
\subsection*{Notations}
The euclidian norm of $x\in \R^d$ will be denoted by $|x|$. By $B(R)$ we denote the open ball in $\R^d$ with centre
at the origin and radius $R$. For $T>0$ we use the notations:
\begin{eqnarray*}
\Lambda_T &:=& \{(t,s): 0\le s\le t\le T\},\\
\widetilde{\Lambda}_T &:=&  \{(t,s): 0\le s< t\le T\},\\
\Lambda &:=& \{(t,s): 0\le s\le t \},\\
\widetilde{\Lambda} &:=&  \{(t,s): 0\le s< t \}.
\end{eqnarray*}
Let us come to notation for function spaces. For a $C^{1,1}$ domain $\Omega\subset \R^d$ and $1\le p< \infty,\,j\in \N$, $W^{j,p}(\Omega)$ denotes the
classical Sobolev space of all $L^p(\Omega)$--functions having weak derivatives in
$L^p(\Omega)$ up to the order $j$. Its usual norm is denoted by $\|\cdot
\|_{j,p,\Omega}$ and by $\|\cdot \|_{p,\Omega}$ when $j=0$. If $\Omega =\R^d$ we drop $\Omega$ in the notations of the above norms.
We will use also the notation
$$\langle f,g\rangle_\Omega :=\int_\Omega fg\,dx,\quad f\in L^p(\Omega),\,g\in L^{p'}(\Omega)$$ with $\frac{1}{p}+\frac{1}{p'}=1$.
By $W_0^{s,p}(\Omega)$, $s\geq 0$, we denote the closure of the space of test functions $C_c^\infty(\Omega)$ with respect to the norm of $W^{s,p}(\Omega)$. For $s<0$ we set
\begin{equation*}
W^{s,p}_0(\Omega) := \left( W^{-s,p'}(\Omega) \right)',
\end{equation*}
where $\frac{1}{p}+\frac{1}{p'}=1$. We denote by $H^{s,p}(\Omega)$ with $s\in(0,2)$ the Bessel potential spaces, which are defined by complex interpolation
$$
H^{s,p}(\Omega):=[L^p(\Omega), W^{2,p}(\Omega)]_{\frac s 2}.
$$
Its norm will be denoted by $\| \cdot \|_{s,p,\Omega}$.
Moreover, we set
\begin{align*}
C_{c,\sigma}^{\infty}(\Omega)&:= \{f\in C_c^{\infty}(\Omega)^d: \; \div f = 0\},\\
L^p(\Omega)^d&:= \{f=(f_1,\ldots,f_d) : f_i \in L^p(\Omega), \, i=1,\ldots,d\}, \quad L_\sigma^p(\Omega):= \overline{C_{c,\sigma}^{\infty}(\Omega)}^{L^p(\Omega)^d},\\
G^p(\Omega)&:= \{\nabla \p : \,\p \in \hat{W}^{1,p}(\Omega)\}, \quad \hat{W}^{1,p}(\Omega):=\{\p \in L^p_{loc}(\overline{\Omega}) : \, \nabla \p \in L^p(\Omega)^d\}.
\end{align*}
It is well-known that for a Lipschitz domain $\Omega\subset \R^d$ with compact boundary the Helmholtz decomposition holds (see e.g. \cite{Galdi:1994} for more information):
\begin{equation*}
L^p(\Omega)^d=L^p_\sigma(\Omega) \,\oplus\,G^p(\Omega).
\end{equation*}
The projection from $L^p(\Omega)^d$ onto $L^p_\sigma(\Omega)$ is denoted by $\P_\Omega$.
%
%
\section{Main results and strategy of proofs}

In this section we present the main results and sketch the basic strategy of the proofs. In the following $\mathcal O \subset \R^d$ is always a compact obstacle with $C^{1,1}$-boundary and $\Omega:=\R^d\setminus \mathcal O$ is an exterior domain. Moreover, $M \in C^1([0,\infty),\R^{d\times d})$ and $v_\infty \in C^1([0,\infty),\R^d)$ are as described in the Introduction. In particular recall that $\mathrm{tr} \;M(t)=0$ for all $t\geq0$. To simplify our notation, we set $c(t):= - Q(t)^{\mathrm T}v_{\infty}(t)$. Since the term $Q(t)\partial_t v_\infty(t)$ in equation \eqref{eq:NS_3} is constant in space, we may put this term in the pressure $\p$. Thus, in the following we consider the system
\begin{align}\label{eq:NS_4}
\left.\begin{array}{l}
u_t- \Delta u - (M(t)x+c(t)) \cdot \nabla u + M(t)u \\
\quad +u\cdot \nabla u +\nabla \p
\end{array}\right\}&=0&\;\mbox{in   $(0,\infty)\times \Omega $,}\notag\\
\div u&=0&\quad\mbox{in $ (0,\infty)\times \Omega $,}\notag\\
u(t,x)&= M(t)x+c(t)&\;\mbox{on $(0,\infty)\times \partial\Omega $},\\
\lim_{|x|\to \infty} u(t,x)&= 0&\;\mbox{\mbox{for} $t\in(0,\infty) $,}\notag\\
u(0,x)&=u_0(x)&\;\mbox{in $\Omega$},\notag
\end{align}
\normalsize
where $\div u_0 = 0$. Moreover, we assume that the initial value $u_0$ satisfies the
compatibility assumption $u_0 \cdot \nu = (M(0)x+c(0)) \cdot \nu$ on $\partial \Omega$, where $\nu$ denotes the outer normal vector.

As a first step we construct a solenoidal extension in $\Omega$ of the boundary velocity $u(t,x)\vert_{\partial \Omega}$. For this
purpose we introduce the Bogovskii operator, which
concerns the solution of the equation $\div u = f$ in appropriate function spaces. This operator will also be needed later in Section
\ref{sect:exterior} to keep the solenoidal condition in our cut-off procedure. For proofs and more information on the Bogovskii operator we refer to \cite{Bogovskii:1979, Geissert/Heck/Hieber:2006b} and to the monograph \cite{Galdi:1994}.
\begin{lemma}\label{prop_Bog1}
Let $D \subset \R^d$, $d\geq 2$, be a bounded Lipschitz domain, $1<p<\infty$ and $k\in \N_0$.
\begin{itemize}
\item[(a)] There exists a continuous operator
\begin{equation*}
\BB_D : W^{k,p}_0(D)\to (W^{k+1,p}_0(D))^d
\end{equation*}
such that
\begin{equation*}
\div \BB_D f = f
\end{equation*}
for all $f\in W^{k,p}_0(D)$ satisfying $\int_D f \,\d x= 0$.
\item[(b)] For $k>-2+\frac 1 p$, the above operator $\BB_D$ can be continuously extended to a bounded operator from $W^{k,p}_0(D)$ to $(W^{k+1,p}_0(D))^d$.
\end{itemize}
\end{lemma}
Let $\zeta \in C_c^\infty(\R^d)$ be a cut-off function with $0\leq \zeta \leq 1$ and $\zeta = 1$ near $\partial \Omega$. Moreover we set $K:= \mathrm{supp} \, \nabla \zeta$ and define $b:[0,\infty)\times \R^d\rightarrow \R^d$ by
\begin{equation*}
b(t,x):=\zeta(x)(M(t)x + c(t)) - \BB_K((\nabla \zeta)\cdot (M(t)x+c(t))),
\end{equation*}
where $\BB_K$ is the operator from Lemma \ref{prop_Bog1} associated to the bounded domain $K$. Then $\div b(t,x)=0$ and
$b(t,x)=M(t)x + c(t)$ on $\partial \Omega$ for every $t\ge0$.

If we set $\tilde u = u - b$, then $u$ satisfies \eqref{eq:NS_4} if and only if $\tilde u$ satisfies
\begin{align}\label{eq:NS_5}
	\left.
	\begin{array}{l}
u_t- \Delta u - (M(t)x+c(t)) \cdot \nabla u + M(t)u \\
\quad + b(t,x)\cdot \nabla u + u\cdot \nabla b(t,x) +u\cdot \nabla u
+\nabla \p
\end{array}\right\}&=F_1(t,x)&\;\mbox{in   $(0,\infty)\times \Omega $,}\notag\\
\div u&=0&\quad\mbox{in $ (0,\infty)\times \Omega $,}\notag\\
u(t,x)&= 0&\;\mbox{on $(0,\infty)\times \partial\Omega $},\\
\lim_{|x|\to \infty} u(t,x)&= 0&\;\mbox{\mbox{for} $t\in(0,\infty) $,}\notag\\
u(0,x)&= f&\;\mbox{in $\Omega$},\notag
\end{align}
\normalsize
where $f(x):= u_0(x)-b(0,x)$ and
\begin{equation}\label{eq:inhom}
F_1(t,x):= \Delta b(t,x) + (M(t)x + c(t))\cdot \nabla b(t,x) - M(t) b(t,x) - b(t,x) \cdot \nabla b(t,x)-b_t(t,x).
\end{equation}
Note that $\div f = 0$ and that the compatibility assumption ensures that even $f\in L^p_\sigma (\Omega)$.

Our approach to system \eqref{eq:NS_5} is based on linear operators of the form
\begin{equation}
\OU (t)u(x) = \Big( \Delta u_i(x) + \langle M(t)x + c(t) , \nabla u_i(x) \rangle \Big)_{i=1}^{d} - M(t)u(x), \quad t\ge0,
\end{equation}
and perturbations of the form
\begin{equation*}
\mathcal B(t)u(x) = -b(t,x)\cdot \nabla u - u\cdot \nabla b(t,x), \quad t\ge0,
\end{equation*}
where $u=(u_1,\ldots,u_d)$ and $x$ is an element from $\Omega$, a bounded domain $D\subset \R^d$ or $\R^d$.  Note that the operators $\OU(\cdot)$ are of Ornstein-Uhlenbeck type (cf. \cite{Hansel/Rhandi:2010}). In the case of exterior domains $\Omega$ we define the $L^p$-realizations of $\OU(\cdot)$ as
\begin{equation}
\begin{array}{rcl}
\D(L_{\Omega}(t))&:=&\{u \in W^{2,p}(\Omega)^d \cap W^{1,p}_0(\Omega)^d: M(t)x \cdot \nabla u \in L^p(\Omega)^d\},\\[0.15cm]
L_{\Omega}(t)u & := & \OU(t)u,
\end{array}
\end{equation}
and the perturbed operators are defined by
\begin{equation}
\begin{array}{rcl}
\D(L_{\Omega,b}(t))&:=&\D(L_\Omega(t)),\\[0.15cm]
L_{\Omega,b}(t)u & := & \OU(t)u + \mathcal B(t) u.
\end{array}
\end{equation}
In the following, to simplify our notation, we do not distinguish between $L^p(\Omega)^d$ and $L^p(\Omega)$ and sometimes write
\begin{equation*}
L_\Omega(t)u(x)  := \Delta u(x)+\left(M(t)x + c(t)\right)\cdot\nabla u(x)
	-M(t)u(x),\quad t>0,\,x\in \Omega.
\end{equation*}
With these linear operators the linearization of the system \eqref{eq:NS_5} for some initial time $s\geq 0$ is now given by
\begin{equation}\label{eq:S_exterior}
\left\{\begin{array}{rcll}
u_t - L_{\Omega,b}(t)u + \nabla \p &=& 0, & \quad \mbox{in}\; (s,\infty)\times \Omega ,\\
\div u &=& 0,&  \quad \mbox{in}\; (s,\infty)\times \Omega ,\\
u &=& 0, & \quad \mbox{on}\; (s,\infty)\times \partial\Omega ,\\
u(s,\cdot) &=&f, &  \quad \mbox{in}\; \Omega.
\end{array}\right.
\end{equation}
As usual in the theory of the Navier-Stokes equations we shall later work in the space $L^p_\sigma(\Omega)$ of all solenoidal vector fields in $L^p$. Therefore we set
\begin{equation}
\begin{array}{rcl}
\D(A_{\Omega}(t))&:=&\{u \in W^{2,p}(\Omega)^d \cap W^{1,p}_0(\Omega)^d \cap L_\sigma^p(\Omega) : M(t)x \cdot \nabla u \in
L^p(\Omega)^d\},\\[0.15cm]
A_{\Omega}(t)u & := &\P_\Omega L_{\Omega}(t)u,
\end{array}
\end{equation}
and
\begin{equation}
\begin{array}{rcl}
\D(A_{\Omega,b}(t))&:=&\D(A_{\Omega}(t)),\\[0.15cm]
A_{\Omega,b}(t)u & := &\P_\Omega L_{\Omega,b}(t)u.
\end{array}
\end{equation}
By applying the Helmholtz projection $\P_\Omega$ to \eqref{eq:S_exterior} the pressure $\p$ can be eliminated and we may rewrite the equations as
an non-autonomous abstract Cauchy problem
\begin{equation}\label{eq:nACP_exterior}
\left\{\begin{array}{lcll}
u'(t) & = & A_{\Omega,b}(t)u(t),& 0\leq s < t ,\\[0.15cm]
u(s) & = & f.
\end{array}\right.
\end{equation}
It directly follows from \cite{Geissert/Heck/Hieber:2006a} that for fixed $s\ge 0$ the operator $A_{\Omega,b}(s)$ generates a $C_0$-semigroup
on $L^p_\sigma(\Omega)$, $1<p<\infty$, which is however \emph{not analytic}. Therefore we cannot apply standard generation results for evolution
systems of parabolic type (we refer to the monographs \cite{Lunardi:1995} and \cite{Tanabe:1997} for more information on this matter). Moreover, we note that the domain of $A_{\Omega,b}(t)$ depends on time $t$. Therefore, to
overcome this difficulty and in order to discuss well-posedness of \eqref{eq:nACP_exterior} we introduce the regularity space
\begin{equation*}
Y_{\Omega}:=\{u \in W^{2,p}(\Omega)^d \cap W^{1,p}_0(\Omega)^d \cap L^p_\sigma(\Omega) : |x| \nabla u_i(x) \in
L^p(\Omega)^d \;
\mathrm{for} \; i=1,\ldots, d\}
\end{equation*}
which is contained in $\D(A_{\Omega,b}(t))$ for every $t\ge0$.

Our first main result is the existence of a strongly
continuous evolution system on $L^p_\sigma(\Omega)$, $1<p<\infty$, that solves the
Cauchy problem \eqref{eq:nACP_exterior} on the regularity space $Y_\Omega$. This directly implies well-posedness of
\eqref{eq:nACP_exterior}. Moreover, we obtain
$L^p$-$L^q$ smoothing properties and gradient estimates for the evolution system. This is \emph{a priori} not obvious, since the
evolution system is not of parabolic type.
\begin{theorem}\label{thm:evolution_system_exterior}
Let $\Omega \subset \R^d$ be an exterior domain with $C^{1,1}$-boundary and $1<p<\infty$. Then there exists a unique evolution system $\{T_{\Omega,b}(t,s)\}_{(t,s)\in \Lambda}$ on $L^p_\sigma(\Omega)$ with the following properties.
\begin{itemize}
\item[(a)] For $(t,s)\in \Lambda$, the operator $T_{\Omega,b}(t,s)$ maps $Y_{\Omega}$ into $Y_{\Omega}$.
\item[(b)] For every $f\in Y_{\Omega}$ and $s\ge 0$, the map $t\mapsto T_{\Omega,b}(t,s)f$ is differentiable in $(s,\infty)$ and
\begin{equation}\label{eq:derivative_evolution_exterior}
\frac{\partial}{\partial t} T_{\Omega,b}(t,s)f = A_{\Omega,b}(t)T_{\Omega,b}(t,s)f .
\end{equation}
\item[(c)] For every $f\in Y_{\Omega}$ and $t>0$, the map $s\mapsto T_{\Omega,b}(t,s)f$ is differentiable in $[0,t)$ and
\begin{equation}\label{eq:derivative_evolution_exterior_2}
\frac{\partial}{\partial s} T_{\Omega,b}(t,s)f = - T_{\Omega,b}(t,s)A_{\Omega,b}(s)f.
\end{equation}
\item[(d)] For $T>0$ and $1<p<\infty$ there is a constant $C:=C(T)>0$ such that for every $f\in L^p_
\sigma(\Omega)$ $$\|\nabla T_{\Omega,b}(t,s)f\|_{p,\Omega} \leq C(t-s)^{-\frac {1}{2}}\|f\|_{p,\Omega}, \qquad (t,s)\in\tilde\Lambda_T.$$
\item[(e)] Let $T>0$ and $1< p\leq q < \infty$. Then there exists a constant $C:=C(T)>0$ such that for every $f\in L^p_
\sigma(\Omega)$
\begin{equation}\label{Lpq}
\|T_{\Omega,b}(t,s)f\|_{q,\Omega} \leq C(t-s)^{-\frac{d}{2}\left(\frac 1 p - \frac 1 q\right)}\|f\|_{p,\Omega}, \qquad
(t,s)\in\tilde\Lambda_T,
\end{equation}
\begin{equation}\label{nabla-Lpq}
\|\nabla T_{\Omega,b}(t,s)f\|_{q,\Omega} \leq C(t-s)^{-\frac{d}{2}\left(\frac 1 p - \frac 1 q\right)-\frac{1}{2}}\|f\|_{p,\Omega}, \qquad (t,s)\in\tilde\Lambda_T.
\end{equation}
\end{itemize}
\end{theorem}
\begin{remark}
\begin{itemize}
\item[(a)] The analogous result holds for the evolution system $\{T_\Omega(t,s)\}_{(t,s)\in\Lambda}$ associated to the operators $A_\Omega(\cdot)$.
\item[(b)] If we denote the evolution system on $L^p_\sigma(\Omega)$ by $\{T^p_{\Omega,b}(t,s)\}_{(t,s)\in \Lambda}$, then the family of evolution systems is consistent in the sense that
\begin{equation*}
T^p_{\Omega,b}(t,s) f = T^q_{\Omega,b}(t,s) f, \qquad \quad (t,s)\in \Lambda, \qquad \quad f \in L^p_\sigma(\Omega)\cap L^q_\sigma (\Omega),
\end{equation*}
holds for $1<p,q<\infty$.
\end{itemize}
\end{remark}
The basic idea to prove Theorem \ref{thm:evolution_system_exterior} is to study first the whole space case $\R^d$ and the case of a
bounded domain $D$, which is
done in Sections \ref{sect_whole} and \ref{sect_bounded} respectively. Then in Section \ref{sect:exterior} we use some cut-off techniques to construct the evolution system for the exterior domain $\Omega$. Our method, which was already presented in \cite{Hansel/Rhandi:2010} for treating non-autonomous Ornstein-Uhlenbeck equations, then also allows to obtain the $L^p$-$L^q$ smoothing properties and the gradient estimates from the corresponding estimates in $\R^d$ and $D$.

Next we come back to the nonlinear problem \eqref{eq:NS_5}. Again by applying the Helmholtz projection $\P_\Omega$ we can rewrite this system in abstract form as
\begin{equation}\label{eq:NS_abstract}
\left\{
\begin{array}{rclll}
u'(t)- A_{\Omega,b}(t)u(t) + \P_\Omega( u \cdot \nabla u)(t) &=&\P_\Omega F_1(t), & t>0,  \\[0.15cm]
u(0)&=&f,&
\end{array}\right.
\end{equation}
where $F_1$ is given in \eqref{eq:inhom}. By the Duhamel principle (variation of constant formula) this problem can be reduced to the
integral equation
\begin{align}\label{eq:Duhamel}
&u(t) = T_{\Omega,b}(t,0)f - \int_0^t T_{\Omega,b}(t,s) \mathbb P_\Omega(u\cdot \nabla u)(s) \d s\\
&\qquad\qquad\qquad +  \int_0^t T_{\Omega,b}(t,s) \mathbb P_\Omega F_1(s)  \d s,\quad \qquad t\geq0,\notag
\end{align}
in $L^p_{\sigma}(\R^d)$.
For given $0<T_0\leq \infty$,
we call a function $u\in C([0,T_0);L^p_{\sigma}(\R^d))$ a \textit{mild solution} of (\ref{eq:NS_abstract})  if $u$ satisfies
the integral equation (\ref{eq:Duhamel}).
By adjusting Kato's iteration scheme 
to our situation the existence of a unique local mild solution follows.
\begin{theorem}\label{thm:Kato}
Let $2\leq d\leq p \le q < \infty$ and $f\in
L^p_{\sigma}(\Omega)$. Then there exists $T>0$ and a unique mild
solution $u\in C([0,T];L^p_{\sigma}(\Omega))$ of
(\ref{eq:NS_abstract}), which has the properties
\begin{equation}\label{eq:mild_solution_prop1}
t^{\frac{d}{2}\left(\frac{1}{p}-\frac{1}{q}\right)}u(t) \in
C([0,T];L^q_{\sigma}(\Omega)),
\end{equation}
\begin{equation}\label{eq:mild_solution_prop2}
t^{\frac{d}{2}\left(\frac{1}{p}-\frac{1}{q}\right)+\frac{1}{2}}\nabla
u(t) \in C([0,T];L^p(\Omega)^{d\times d}).
\end{equation}
If $p<q$, then in addition
\begin{equation}\label{eq:mild_solution_prop3}
t^{\frac{d}{2}\left(\frac{1}{p}-\frac{1}{q}\right)}\|u(t)\|_{q,\Omega} +
t^{\frac{1}{2}}\|\nabla u(t)\|_{p,\Omega} \rightarrow 0 \qquad \mbox{as\;} t
\rightarrow 0^+.
\end{equation}
\end{theorem}

\section{The linearized problem in $\R^d$}\label{sect_whole}
In this section we study the linearized problem in the whole space $\R^d$. This situation was already studied in detail by
the first author in \cite{Hansel:2009} and by Geissert and the first author in \cite{Geissert/Hansel:2010}. Here we recall the main results. For this purpose we set
\begin{equation}\label{eq:operator_Rd}
\begin{array}{rcl}
\D(L_{\R^d}(t))&:=& \{ u\in W^{2,p}(\R^d)^d: M(t)x\cdot
	\nabla u\in L^p(\R^d)^d\},\\[0.2cm]
L_{\R^d}(t)u & := & \OU(t)u.
\end{array}
\end{equation}
and
\begin{equation}
\begin{array}{rcl}
\D(A_{\R^d}(t))&:=& \D(L_{\R^d}(t)) \cap L^p_\sigma (\R^d),\\[0.2cm]
A_{\R^d}(t)u & := & \OU(t)u.
\end{array}
\end{equation}
Note that $A_{\R^d}(t)$ is indeed an operator on $L^p_\sigma(\R^d)$, since an easy computation yields that $\div A_{\R^d}(t)u = 0$ for all $u\in C_{c,\sigma}^\infty(\R^d)$. In particular, the operator $L_{\R^d}(t)$ commutes with the Helmholtz projection $\P_{\R^d}$ in $\R^d$. Moreover we introduce the regularity space
\begin{equation*}
Y_{\R^d}:=\{u \in W^{2,p}(\R^d)^d \cap L^p_\sigma(\R^d) : |x| \nabla u_i(x) \in L^p(\R^d)^d \; \mathrm{for} \; i=1,\ldots, d\}.
\end{equation*}
For every $t\ge0$, the space $Y_{\R^d}$ is contained in $\D(A_{\R^d}(t))$.

In the following we denote by $\{U(t,s)\}_{t,s\ge 0}$ the evolution system in $\R^d$ that satisfies
\begin{equation*}
\left\{\begin{array}{lcll}
\frac{\partial}{\partial t} U(t,s) & = & -M(t)U(t,s),& t\ge s,\\[0.2cm]
U(s,s) & = & \mathrm{Id}.
\end{array}\right.
\end{equation*}
Now, for $f\in L^p_\sigma(\R^d)$ and $s\geq 0$, we set $T_{\R^d}(s,s)=\mathrm{Id}$ and for $(t,s)\in \widetilde{\Lambda}$ we define
\begin{equation}\label{eq:evol_sys_whole}
T_{\R^d}(t,s)f(x) = (k(t,s,\cdot)\ast f)(U(s,t)x+g(t,s)) \qquad x\in \R^d,
\end{equation}
where
\begin{equation}
k(t,s,x):= \frac{1}{(4\pi)^{\frac d 2} (\det Q_{t,s})^{\frac 1 2}} U(t,s) \mathrm{e}^{-\frac 1 4 \langle Q_{t,s}^{-1}x,x\rangle}, \qquad x\in \R^d,
\end{equation}
\begin{equation}
g(t,s) = \int_s^t U(s,r)c(r) \d r \qquad \mbox{and} \qquad Q_{t,s}=\int_s^t U(s,r)U^*(s,r)\d r.
\end{equation}

For the derivation of this solution formula we refer to \cite[Section 3]{Geissert/Hansel:2010}. The explicit formula now allows to prove the following result (see \cite{Geissert/Hansel:2010, Hansel:2009} and \cite[Section 2]{Hansel/Rhandi:2010} for details).
\begin{prop}\label{prop:evolution_whole}
Let $1<p<\infty$. Then the family of operators $\{T_{\R^d}(t,s)\}_{(t,s)\in \Lambda}$ defined in \eqref{eq:evol_sys_whole} is a strongly continuous evolution system on $L^p_\sigma(\R^d)$ with the following properties.
\begin{itemize}
\item[(a)] For $(t,s)\in \Lambda$, the operator $T_{\R^d}(t,s)$ maps $Y_{\R^d}$ into $Y_{\R^d}$.
\item[(b)] For every $f\in Y_{\R^d}$ and every $s\in[0,\infty)$, the map $t\mapsto T_{\R^d}(t,s)f$ is differentiable in $(s,\infty)$ and
\begin{equation}\label{eq:derivative_t}
\frac{\partial}{\partial t} T_{\R^d}(t,s)f = A_{\R^d}(t)T_{\R^d}(t,s)f .
\end{equation}
\item[(c)] For every $f\in Y_{\R^d}$ and $t\in (0,\infty)$, the map $s\mapsto T_{\R^d}(t,s)f$ is differentiable in $[0,t)$
and
\begin{equation}\label{eq:derivative_s}
\frac{\partial}{\partial s} T_{\R^d}(t,s)f = -T_{\R^d}(t,s)A_{\R^d}(s)f .
\end{equation}
\item[(d)] Let $1<p\leq q \leq \infty$. Then there exists a constant
	$C>0$ such that for every $f\in L^p_\sigma(\R^d)$
\begin{align}
	\|T_{\R^d}(t,s)f\|_{q}
	&\leq C
(t-s)^{-\frac{d}{2}\left(\frac{1}{p}-\frac{1}{q}\right)}\|f\|_{p},&
 (t,s)\in \widetilde{\Lambda},\label{eq:LpLqsmoothing}\\
\|\nabla T_{\R^d}(t,s)f\|_{q}
&\leq C
(t-s)^{-\frac{d}{2}\left(\frac{1}{p}-\frac{1}{q}\right)-\frac{1}{2}}\|f\|_{p},
&  (t,s)\in \widetilde{\Lambda}. \label{eq:GradientEstimate}
\end{align}
\end{itemize}
\end{prop}
%
%
By Proposition \ref{prop:evolution_whole} and the fact that $L_{\R^d}(t)=A_{\R^d}(t)$ on $L^p_\sigma (\R^d)$ it follows that the solution to the problem
\begin{equation}\label{eq:problem_bounded}
\left\{\begin{array}{rcll}
u_t - L_{\R^d}(t)u &=& 0, & \quad \mbox{in}\; (s,\infty)\times \R^d ,\\
\div u &=& 0,&  \quad \mbox{in}\; (s,\infty)\times \R^d ,\\
u(s,\cdot) &=&f, &  \quad \mbox{in}\; \R^d,
\end{array}\right.
\end{equation}
for $s\geq 0$, is given by $u(t,x)=T_{\R^d}(t,s)f(x)$. So in the whole space case the pressure is constant, thus we may assume that the pressure is even zero.
\section{The linearized problem in bounded domains}\label{sect_bounded}
In this section let $D\subset \R^d$ be a bounded domain with $C^{1,1}$-boundary. We set
\begin{equation}\label{eq:operator_bounded}
\begin{array}{rcl}
\D(L_{D,b}(t))&:=& \D(L_{D,b}) := W^{2,p}(D)^d\cap W^{1,p}_0(D)^d,\\[0.2cm]
L_{D,b}(t)u & := & \OU_{D,b}(t)u.
\end{array}
\end{equation}
Moreover, we define the operators
\begin{equation}\label{eq:operator_bounded}
\begin{array}{rcl}
\D(A_{D,b}(t))&:=& \D(A_{D,b}) := W^{2,p}(D)^d\cap W^{1,p}_0(D)^d \cap L^p_\sigma(D),\\[0.2cm]
A_{D,b}(t)u & := & \P_D L_{D,b}(t)u.
\end{array}
\end{equation}
In a bounded domain also the coefficients of the term $M(t)x\cdot \nabla$ are bounded. Thus, it follows directly from the classical perturbation theory for the Stokes operator that for fixed $s\ge 0$ the operator $(A_{D,b}(s),\D(A_{D,b}))$ generates an analytic semigroup (cf. \cite[Proposition 3.1]{Geissert/Heck/Hieber:2006a}). Moreover, by our assumptions on the coefficients we obtain that the map $t\mapsto A_{D,b}(t)$ belongs to $C^1(\R_+, \L(\D(A_{D,b}),L^p_\sigma(D)))$. Thus, the following result follows from the theory of parabolic evolution systems (see \cite[Chapter 6]{Lunardi:1995} and \cite[Section 2.3]{Grisvard}).
\begin{prop}\label{prop:evolution_bounded}
Let $D\subset \R^d$ be a bounded domain with $C^{1,1}$-boundary and $1<p<\infty$. Then there is a unique evolution system $\{T_{D,b}(t,s)\}_{(t,s)\in \Lambda}$ on $L^p_\sigma(D)$ with the following properties.
\begin{enumerate}
\item[(a)] For $(t,s)\in \widetilde{\Lambda}$, the operator $T_{D,b}(t,s)$ maps $L^p_\sigma(D)$ into $\D(A_{D,b})$.
\item[(b)] The map $t \mapsto T_{D,b}(t,s)$ is differentiable in $(s,\infty)$ with values in $\L(L^p_\sigma(D))$ and
\begin{equation}
\frac{\partial}{\partial t} T_{D,b}(t,s) = A_{D,b}(t)T_{D,b}(t,s).
\end{equation}
\item[(c)] For every $f \in \D(A_{D,b})$ and $t>0$, the map $s\mapsto T_{D,b}(t,s)f$ is differentiable in $[0,t)$ and
\begin{equation}
\frac{\partial}{\partial s} T_{D,b}(t,s)f = - T_{D,b}(t,s)A_{D,b}(s)f.
\end{equation}
\item[(d)] Let $T>0$. Then there exists a constant $C:=C(T)>0$ such that
\begin{equation}\label{eq45}
\|T_{D,b}(t,s)f\|_{p,D}\leq C\|f\|_{p,D},
\end{equation}
and
\begin{equation}\label{eq46}
\|T_{D,b}(t,s)f\|_{2,p,D}\leq C(t-s)^{-1}\|f\|_{p,D}.
\end{equation}
for all $f\in L^p_\sigma(D)$ and all $(t,s)\in \widetilde{\Lambda}_T$.
\end{enumerate}
\end{prop}
For the following first estimate let us recall the Gagliardo-Nierenberg inequality
\begin{equation}\label{G-N}
\|f\|_{k,q,D}\le C \|f\|_{m,p,D}^\theta \|f\|_{r,D}^{(1-\theta)},\quad \hbox{\ for all }f\in W^{m,p}(D),
\end{equation}
where $p\ge 1,\,q\ge 1,\,r\ge 1,\,0<\theta \le 1$ and $k-\frac{d}{q}\le \theta\left(m-\frac{d}{p}\right)-(1-\theta)\frac{d}{r}.$
%
\begin{corollary}\label{prop:Lp_Lq_estimates_bounded}
Let $T>0$ and $1< p\leq q < \infty$. Then there exists a constant $C:=C(T)>0$ such that for every $f\in L^p_\sigma(D)$
and $(t,s)\in \widetilde{\Lambda}_T$
\begin{enumerate}
\item[(a)] $\|T_{D,b}(t,s)f\|_{q,D} \leq C(t-s)^{-\frac{d}{2}\left(\frac 1 p - \frac 1 q\right)}\|f\|_{p,D},$
\item[(b)] $\|\nabla T_{D,b}(t,s)f\|_{p,D} \leq C(t-s)^{-\frac{1}{2}}\|f\|_{p,D},$
\item[(c)]  $\|\nabla T_{D,b}(t,s)f\|_{q,D} \leq C(t-s)^{-\frac{d}{2}\left(\frac 1 p - \frac 1 q\right)-\frac 1 2}\|f\|_{p,D}$.
\end{enumerate}
Moreover, $$\|T_{D,b}(t,s)f\|_{k,p}\le C\|f\|_{k,p},\quad (t,s)\in \Lambda_T,$$
for all $f\in W^{k,p}(D)^d\cap L^p_\sigma(D),\,k=1,\,2,$ and
$$\|T_{D,b}(t,s)f\|_{2,p}\le C(t-s)^{-\frac{1}{2}}\|f\|_{1,p},\quad (t,s)\in \widetilde{\Lambda}_T,$$
for all $f\in W^{1,p}(D)^d\cap L^p_\sigma (D)$.
\end{corollary}
\begin{proof}
Let us assume first that $0<\frac{1}{p}-\frac{1}{q}\le \frac{2}{d}$. Then $0<\frac{d}{2}\left(\frac{1}{p}-\frac{1}{q}\right)\le 1$.
Hence (a) follows by applying (\ref{G-N}) with $\theta=\frac{d}{2}\left(\frac{1}{p}-\frac{1}{q}\right),\,k=0,\,m=2,\,r=p$ and
(\ref{eq46}).\\
Assume now that $\frac{2}{d}<\frac{1}{p}-\frac{1}{q}\le \frac{4}{d}$. Set $\frac{1}{r}=\frac{1}{q}+\frac{2}{d}$. Then
$0<\frac{1}{p}-\frac{1}{r}\le \frac{2}{d}$. So, by the first step and (\ref{G-N}), we obtain
\begin{eqnarray*}
\|T_{D,b}(t,s)f\|_{q,D} &=& \left\|T_{D,b}\left(t,s+\tfrac{t-s}{2}\right)T_{D,b}\left(s+\tfrac{t-s}{2},s\right)f\right\|_{q,D}\\
&\le & C\left(\frac{t-s}{2}\right)^{-\frac{d}{2}\left(\frac 1 r - \frac 1 q\right)}\left\|T_{D,b}\left(s+\tfrac{t-s}{2},s\right)f\right\|_{r,D}\\
&\le & C\left(\frac{t-s}{2}\right)^{-\frac{d}{2}\left(\frac 1 r - \frac 1 q\right)}\left(\frac{t-s}{2}\right)^{-\frac{d}{2}\left(\frac 1 p - \frac 1 r\right)}\|f\|_{p,D}\\
&\le & C(t-s)^{-\frac{d}{2}\left(\frac 1 p - \frac 1 q\right)}\|f\|_{p,D},\quad (t,s)\in \widetilde{\Lambda}_T.
\end{eqnarray*}
By iterating this argument we obtain (a) also for $\frac{1}{p}-\frac{1}{q}>\frac{2}{d}$.
\\
Assertion (b) follows from (\ref{G-N}) with $\theta=\frac{1}{2},\,p=q=r,\,k=1,\,m=2$ and (\ref{eq46}).
\\
It follows from (a) and (b) that
\begin{eqnarray*}
\|\nabla T_{D,b}(t,s)f\|_{q,D} &=& \left\| \nabla T_{D,b}\left(t,s+\tfrac{t-s}{2}\right)T_{D,b}\left(s+\tfrac{t-s}{2},s\right)f\right\|_{q,D}\\
&\le & C\left(\frac{t-s}{2}\right)^{-\frac{1}{2}}\left\|T_{D,b}\left(s+\tfrac{t-s}{2},s\right)f\right\|_{q,D}\\
&\le & C(t-s)^{-\frac{d}{2}\left(\frac 1 p - \frac 1 q\right)-\frac 1 2}\|f\|_{p,D}, \quad (t,s)\in \widetilde{\Lambda}_T,
\end{eqnarray*}
and this proves (c).

For the last assertions we refer, for example, to \cite[Corollary 6.1.8]{Lunardi:1995}.
\end{proof}
By Proposition \ref{prop:evolution_bounded} the solution to the problem
\begin{equation}\label{eq:problem_bounded}
\left\{\begin{array}{rcll}
u_t - L_{D,b}(t)u + \nabla \p &=& 0, & \quad \mbox{in}\; (s,\infty)\times D ,\\
\div u &=& 0,&  \quad \mbox{in}\; (s,\infty)\times D ,\\
u &=& 0, & \quad \mbox{on}\; (s,\infty)\times \partial D ,\\
u(s,\cdot) &=&f, &  \quad \mbox{in}\; D,
\end{array}\right.
\end{equation}
for $s\geq 0$, is given by $u(t,x)=T_{D,b}(t,s)f(x)$. From \eqref{eq:problem_bounded} and the fact that $u_t \in L^p_\sigma(D)$ it follows that $\nabla \p(t) = (\mathrm{Id}-\P_D)L_{D,b}(t)u(t)$.  Since the pressure is only unique up to an additive constant, in the case of a bounded domain we can always assume that
$$
\p \in L^p_0(D) :=\left\{u \in L^p(D): \int_D u \,\d x = 0\right\}.
$$
By the abstract theory of parabolic evolution systems it is clear that $u\in C^1((s,\infty); L^p_\sigma(D)) \cap C((s,\infty); \D(A_{D,b}))$. So in particular, we can conclude that $\nabla \p \in C((s,\infty),L^p(D)^d)$. By Poincar\'{e}'s inequality we can conclude that also $\p \in C((s,\infty),L^p(D))$.
\begin{lemma}\label{lem:pressure}
Let $1<p<\infty$ and $D\subset \R^d$ be a bounded domain with $C^{1,1}$-boundary. Let $(u,\p)$ be the unique solution of
\eqref{eq:problem_bounded} with $\p\in L^p_0(D)$ and let $\gamma \in (1+\frac 1 p,2)$. Then for $T>s$ there exists a
constant $C:=C(T,\gamma)>0$ with
\begin{equation*}
\|\p(t)\|_{p,D}\leq C(t-s)^{-\frac \gamma 2}\|f\|_{p,D}
\end{equation*}
for all $t\in (s,T)$.
\end{lemma}
A similar estimate was proved in \cite[Lemma 3.5]{Geissert/Heck/Hieber:2006a} for the solution to the corresponding resolvent problem (see also \cite{Noll/Saal:2003} for the case of the Stokes operator). However, our proof follows the ideas in \cite{Shibata/Shimada:2007} (see also \cite[Section 4]{Shibata:2008}).
\begin{proof}[Proof of Lemma \ref{lem:pressure}]
Let $\varphi \in C_c^\infty(D)$ and set $\tilde \varphi = \varphi - |D|^{-1}\int_D \varphi\,\d x$. Then, in $L^{p'}(D)$ with
 $\frac{1}{p}+\frac{1}{p'}=1$, there exits a unique solution $\psi \in W^{2,p'}(D)$ of the Neumann problem
\begin{equation}\label{eq:Neumann}
\left\{\begin{array}{rcll}
\Delta \psi &=& \tilde\varphi &\mbox{in}\;D,\\
\nabla \psi \cdot\nu &=&0&\mbox{on}\; \partial D,
\end{array}\right.
\end{equation}
which satisfies the estimate
\begin{equation*}
\|\psi\|_{2,p',D}\leq C \|\tilde\varphi\|_{p',D}\le 2C\|\varphi \|_{p',D},
\end{equation*}
for some constant $C>0$ (cf. \cite[Proposition 5.5]{Shibata/Shimada:2007}). We obtain now
\begin{align*}
\langle \p, \varphi \rangle_D &= \langle \p, \tilde\varphi \rangle_D = \langle \p, \Delta\psi \rangle_D = -  \langle \nabla\p, \nabla\psi \rangle_D=-\langle (\mathrm{Id}-\P_D)L_{D,b}(t)u(t), \nabla\psi \rangle_D\\
&= -\langle \Delta u(t), \nabla\psi \rangle_D-\langle (M(t)x-c(t))\cdot \nabla u(t), \nabla\psi \rangle_D + \langle M(t)u(t), \nabla\psi \rangle_D\\
& \quad +\langle b\cdot \nabla u(t), \nabla \psi \rangle_D +\langle u(t)\cdot \nabla b, \nabla \psi \rangle_D \\
&= -\langle \nabla u(t)\cdot \nu,\nabla \psi \rangle_{\partial D}+ \langle \nabla u(t), \nabla^2\psi \rangle_D
-\langle (M(t)x-c(t))\cdot \nabla u(t), \nabla\psi \rangle_D \\
& \quad + \langle M(t)u(t), \nabla\psi \rangle_D
+\langle b\cdot \nabla u(t), \nabla \psi \rangle_D +\langle u(t)\cdot \nabla b, \nabla \psi \rangle_D.
\end{align*}
By using the embedding $H^{\frac 1 p + \eps, p}(D) \hookrightarrow L^p(\partial D)$ for $0<\eps \leq 1 - \frac 1 p$
(cf \cite[Theorem 1.5.1.2]{Grisvard}) we obtain
\begin{align*}
|\langle \p, \varphi \rangle_D| &\leq C \left( \|\nabla u \cdot \nu\|_{p,\partial D}  +\|\nabla u\|_{p,D}+\| u\|_{p,D}\right)\|\psi\|_{2,p',D}\\
&\leq C\|u \|_{1+\frac 1 p +\eps,p,D} \|\varphi\|_{p',D}.
\end{align*}
By taking now $\eps <1-\frac 1 p$ and $\gamma =1+\frac 1 p +\eps$, the assertion follows from Proposition \ref{prop:evolution_bounded}, Corollary \ref{prop:Lp_Lq_estimates_bounded} and simple interpolation .
\end{proof}
\section{The linearized problem in exterior domains}\label{sect:exterior}
In this section we prove Theorem \ref{thm:evolution_system_exterior}. The general idea is to derive the result for exterior domains from the corresponding results in the
case of $\R^d$ and bounded domains by
some cut-off techniques. For this purpose let $R>0$ be such that $\mathcal O \cup \mathrm{supp}\,b(\cdot,t) \subset B(R)$ for every $t>0$. We then set
\begin{align*}
D&=\Omega \cap B(R+8),\\
K_1 & = \Omega \cap B(R+2),\\
K_2 & = \{ x \in \Omega : R+2 < |x|<R+5\},\\
K_3 & = \{ x \in \Omega : R+5 < |x|<R+8\}.
\end{align*}
We denote by $\BB_i$ for $i\in\{1,2,3\}$ the operator defined in Lemma \ref{prop_Bog1} associated to the domain $K_i$. Moreover, by $\{T_{\R^d}(t,s)\}_{(t,s)\in\Lambda}$ we denote the evolution system in $L^p_\sigma(\R^d)$ from Proposition \ref{prop:evolution_whole} and by $\{T_{D,b}(t,s)\}_{(t,s)\in\Lambda}$
the evolution system in $L^p_\sigma(D)$ for the bounded domain $D$ from Proposition \ref{prop:evolution_bounded}.

Next, we choose cut-off functions $\varphi,\xi,\eta \in C^{\infty}(\R^d)$ such that $0\leq \varphi, \xi, \eta \leq 1$ and
$$
\varphi(x) :=
\left\{
\begin{array}{cc}
  1,&|x| \geq R+4,  \\
  0, & |x| \leq R+3,
\end{array}
\right.
$$
$$
\xi(x) :=
\left\{
\begin{array}{ll}
  1,&|x| \geq R+1,  \\
  0, & |x| \leq R,
\end{array}
\right.
$$
$$
\eta(x) :=
\left\{
\begin{array}{ll}
  1,&|x| \leq R+6,  \\
  0, & |x| \geq R+7.
\end{array}
\right.
$$
For a given $f\in L^p_{\sigma}(\Omega)$ we now define functions $f_R\in L^p_\sigma(\R^d)$ and $f_D\in L^p_\sigma(D)$, respectively, by
$$
f_R(x) :=
\left\{
\begin{array}{ll}
  \xi(x) f(x) - \BB_1((\nabla \xi)\cdot f)(x),&x\in\Omega,   \\
  0, & x\not\in \Omega,
\end{array}
\right.
$$
and
$$
f_D(x) := \eta(x)f(x) - \BB_3((\nabla \eta)\cdot f)(x), \qquad\quad x\in D.
$$
With
partial integration we see that for every $f\in \D(A_{\Omega,b}(t))$ we have $f_R\in \D(A_{\R^d}(t))$ and $f_D\in \D(A_{D,b}(t))$.

Now for $(t,s)\in\Lambda$, we define the operator $W(t,s)$ by  setting
\begin{equation*}
W(t,s)f := \varphi T_{\R^d}(t,s)f_R + (1-\varphi) T_{D,b}(t,s)f_D - \BB_2((\nabla\varphi)\cdot ( T_{\R^d}(t,s)f_R  -T_{D,b}(t,s)f_D))
\end{equation*}
for every $f\in L^p_\sigma(\Omega)$. By Lemma \ref{prop_Bog1} it is clear that $W(t,s)f \in L^p_\sigma(\Omega)$. Moreover, for $f
\in \D(A_{\Omega,b}(t))$ it follows from the properties of $\{T_{\R^d}(t,s)\}_{(t,s)\in\Lambda}$, $\{T_{D,b}(t,s)\}_{(t,s)\in\Lambda}$ and the operator $\BB_2$ that $W(t,s)f \in W^{2,p}(\Omega)\cap W^{1,p}_0(\Omega)\cap L^p_\sigma(\Omega)$
holds. Moreover, a short calculation yields
\begin{align*}
\nabla W(t,s)f & = \varphi \nabla T_{\R^d}(t,s)f_R + (1-\varphi) \nabla T_{D,b}(t,s)f_D \\
&\qquad + \nabla \varphi \left(T_{\R^d}(t,s)f_R - T_{D,b}(t,s)f_D\right)\\
&\qquad - \nabla \left[\BB_2((\nabla \varphi)\cdot (T_{\R^d}(t,s)f_R - T_{D,b}(t,s)f_D) )\right].
\end{align*}
Thus, it follows that $W(t,s)f \in \D(A_{\Omega,b}(t))$ if $f\in \D(A_{\Omega,b}(t))$.

Let us set $u_R(t):=T_{\R^d}(t,s)f_R$ and $u_D(t):=T_{D,b}(t,s)f_D$ and let $\p_D$ be the pressure that is associated
to $u_D$. We may assume that $\int_D \p_D \,\d x= 0$.

A short calculation yields that for some initial value $f\in Y_\Omega$, the function
$u(t):=W(t,s) f$ solves the inhomogeneous equation
\begin{equation}\label{eq:inhom_error_terms}
\left\{\begin{array}{rcll}
u_t - L_{\Omega,b} (t)u +\nabla \p_\varphi &=&- \tilde f, & \quad \mbox{in}\; (s,\infty)\times \Omega ,\\
\div u &=& 0,&  \quad \mbox{in}\; (s,\infty)\times \Omega ,\\
u &=& 0, & \quad \mbox{on}\; (s,\infty)\times \partial\Omega ,\\
u(s,\cdot) &=&f, &  \quad \mbox{in}\; \Omega,
\end{array}\right.
\end{equation}
where $\nabla\p_\varphi := \nabla\left((1-\varphi)\p_D\right)$ and
\begin{align}\label{eq:error_terms}
\tilde f := F(t,s)f & := +2\left(  \nabla  T_{\R^d}(t,s)f_R - \nabla T_{D,b}(t,s)f_D \right)\cdot  (\nabla \varphi)^* \notag \\
& \qquad + (\Delta \varphi + (M(t)x + c(t))\cdot (\nabla \varphi))\left(T_{\R^d}(t,s)f_R - T_{D,b}(t,s)f_D\right)\notag\\
&\qquad +\BB_2((\nabla \varphi)\cdot (\partial_t T_{\R^d}(t,s)f_R  -\partial_t T_{D,b}(t,s)f_D)) \\
&\qquad -  L_{\Omega,b}(t)\BB_2((\nabla \varphi)\cdot (T_{\R^d}(t,s)f_R - T_{D,b}(t,s)f_D) ) \notag\\
&\qquad + (\nabla \varphi) \p_D \notag\\
&=: I_1 + I_2 + I_3 + I_4 + I_5.\notag
\end{align}
Here we use the fact that $\mathrm{supp}\,b(t,\cdot)\subset B(R)$ for every $t>0$ and the expression of $f_R$.

Certainly, the function $F(t,s)f$ in (\ref{eq:error_terms}) is well-defined for every $f\in L^p_\sigma (\Omega)$ and $(t,s)\in \widetilde
\Lambda$. Later, we need a certain decay of $F(t,s)$ in $t$, stated in the following lemma.
\begin{lemma}\label{lemma:error_terms}
Let $1<p<\infty$. For every $f\in L^p_\sigma (\Omega)$ we have
\begin{equation*}
F(\cdot,\cdot)f \in C(\widetilde \Lambda; L^p(\Omega)^d).
\end{equation*}
Moreover, let $T>0$ be fixed and let $\gamma\in(1+\frac 1 p, 2)$. Then
\begin{equation}\label{eq:norm_estimate_error}
\|F(t,s)f\|_{p,\Omega}\leq C  (t-s)^{-\frac \gamma 2} \|f\|_{p,\Omega}, \qquad (t,s)\in \widetilde\Lambda_T,
\end{equation}
for  some constant $C:=C(T,\gamma)>0$.
\end{lemma}
\begin{proof}
Let us start with the norm estimates for $I_1$ and $I_2$.
By using Proposition \ref{prop:evolution_whole} and Corollary \ref{prop:Lp_Lq_estimates_bounded} we obtain
\begin{equation*}
\|I_1\|_{p,\Omega} \leq C (t-s)^{-\frac 1 2} \|f\|_{p,\Omega} \qquad \mbox{and} \qquad \|I_2\|_{p,\Omega} \leq C \|f\|_{p,\Omega}.
\end{equation*}
In order to estimate the norm of $I_4$, let us first note that Lemma \ref{prop_Bog1} implies that $L_{\Omega,b}(t)\BB_2 \in \L(W^{1,p}_0 (K_2), L^p(K_2))$. Thus again by using Proposition \ref{prop:evolution_whole} and Corollary \ref{prop:Lp_Lq_estimates_bounded} we obtain
\begin{equation*}
\|I_4\|_{p,\Omega} \leq C \|f\|_{p,\Omega}.
\end{equation*}
We come now to the term $I_3$. At first we note that we can write
\begin{align*}
&\BB_2((\nabla \varphi)\cdot (\partial_t u_R(t)  -\partial_t u_D(t))) \\
& \qquad = \BB_2((\nabla \varphi)\cdot (L_{\R^d}(t) u_R(t)))  -\BB_2((\nabla\varphi)\cdot (L_{D,b}(t) u_D(t) - \nabla \p_D)).
\end{align*}
Now, for a test function $\psi \in C_c^\infty (\R^d)$, we have
\begin{align*}
&|\langle(\nabla \varphi)\cdot L_{\R^d}(t) u_R(t), \psi\rangle| = |\langle L_{\R^d}(t) u_R(t), \psi (\nabla \varphi) \rangle|\\
&\qquad \quad \leq |\langle \nabla u_R(t), \nabla \left(\psi (\nabla \varphi)\right)  \rangle| +  |\langle (M(t)x+c(t)) \cdot \nabla u_R(t) - M(t)u_R(t), \psi (\nabla \varphi) \rangle|\\
&\qquad \quad \leq C \|u_R(t)\|_{1,p}\|\psi\|_{1,p'} + C \|u_R(t)\|_{1,p,K_2}\|\psi\|_{p',K_2}\\
&\qquad \quad \leq  C   \|u_R(t)\|_{1,p}\|\psi\|_{1,p'},
\end{align*}
where $\frac 1 p + \frac{1}{p'}=1$.  This shows that
\begin{equation*}
\|(\nabla \varphi)\cdot L_{\R^d}(t) u_R(t)\|_{W^{-1,p}(\R^d)}\leq  C   \|u_R(t)\|_{1,p}
\end{equation*}
holds. Analogously, we obtain
\begin{equation*}
\|(\nabla \varphi)\cdot L_{D,b}(t) u_D(t)\|_{W^{-1,p}(D)}\leq  C   \left(\|u_D(t)\|_{1,p,D} + \|\p_D\|_{p,D}\right).
\end{equation*}
Now, Lemma \ref{prop_Bog1} together with Proposition \ref{prop:evolution_whole}, Corollary \ref{prop:Lp_Lq_estimates_bounded} and Lemma \ref{lem:pressure} yield
\begin{equation*}
\|I_3\|_{p,\Omega} \leq C (t-s)^{-{\frac \gamma 2}} \|f\|_{p,\Omega}.
\end{equation*}
From Lemma \ref{lem:pressure} we can conclude that
\begin{equation*}
\|I_5\|_{p,\Omega} \leq C(t-s)^{-\frac \gamma 2}\|f\|_{p,\Omega}.
\end{equation*}
This proves \eqref{eq:norm_estimate_error}.

The continuity of the map $\widetilde{\Lambda}\ni (t,s)\mapsto F(t,s)f$ follows from the strong continuity of $T_{\R^d}(\cdot,\cdot)$ and $T_{D,b}(\cdot,\cdot)$, the continuity of the pressure $\p_D(\cdot)$ together with the properties of the operator $\BB_1$ stated in Lemma \ref{prop_Bog1}.
\end{proof}
Applying $\P_\Omega$ to (\ref{eq:inhom_error_terms}) we have
\begin{equation}\label{eq:inhom_error_terms+}
\left\{\begin{array}{rcll}
u_t &=&  A_{\Omega,b} (t)u - \P_\Omega \tilde f, & \quad \mbox{in}\; (s,\infty)\times \Omega ,\\
\div u &=& 0,&  \quad \mbox{in}\; (s,\infty)\times \Omega ,\\
u &=& 0, & \quad \mbox{on}\; (s,\infty)\times \partial\Omega ,\\
u(s,\cdot) &=&f, &  \quad \mbox{in}\; \Omega .
\end{array}\right.
\end{equation}
It is clear, that if an evolution system $\{T_{\Omega,b}(t,s)\}_{(t,s)\in\Lambda}$ exists on $L^p_\sigma(\Omega)$, then the solution $u(t)$ to the inhomogeneous equation (\ref{eq:inhom_error_terms+}) is given by the variation of constant formula
\begin{equation}\label{eq:integral}
u(t) = T_{\Omega,b}(t,s)f - \int_s^t T_{\Omega,b}(t,r) \P_\Omega F(r,s)f \d r.
\end{equation}
The integral in \eqref{eq:integral} exists because of Lemma \ref{lemma:error_terms}. This consideration suggests to consider the integral equation
\begin{equation}\label{eq:integral_eq}
T_{\Omega,b}(t,s)f = W(t,s)f + \int_s^t T_{\Omega,b}(t,r) \P_\Omega F(r,s)f \d r \qquad (t,s)\in\Lambda, f\in L^p_\sigma(\Omega).
\end{equation}
Let us now state a lemma which will be very useful in the proof of Theorem \ref{thm:evolution_system_exterior}. For the proof we refer to \cite{Hansel/Rhandi:2010}.
\begin{lemma}\label{lemma:iterated}
Let $X_1$ and $X_2$ be two Banach spaces, $T>0$ and let $R:\widetilde{\Lambda}_T \to \L(X_2,X_1)$ and $S:\widetilde{\Lambda}_T \to \L(X_2)$ be strongly continuous functions. Assume that
\begin{equation*}
\|R(t,s)\|_{\L(X_2,X_1)} \leq C_0(t-s)^{\alpha}, \quad \|S(t,s)\|_{\L(X_2)} \leq C_0(t-s)^{\beta}, \quad (t,s)\in \widetilde{\Lambda}_T,
\end{equation*}
holds for some $C_0=C_0(T)>0$ and $\alpha, \beta >-1$. For $f\in X_2$ and $(t,s)\in \widetilde{\Lambda}_T$, set
$T_0(t,s)f:=R(t,s)f$ and
\begin{equation*}
T_n(t,s)f:= \int_s^t T_{n-1}(t,r)S(r,s)f\d s, \qquad n\in \N, \; (t,s)\in \widetilde{\Lambda}_T.
\end{equation*}
Then there exists a constant $C > 0$ such that
\begin{equation}\label{eq:convergence_series}
\sum_{n=0}^\infty \|T_n(t,s)f\|_{X_1} \leq C (t-s)^{\alpha}\|f\|_{X_2}, \qquad (t,s)\in \widetilde{\Lambda}_T.
\end{equation}
Moreover, if $\alpha \geq 0$, the convergence of the series in (\ref{eq:convergence_series}) is uniform on $\Lambda_T$.
\end{lemma}
\begin{proof}[Proof of Theorem  \ref{thm:evolution_system_exterior}.]
Let $T>0$. By using Proposition \ref{prop:evolution_whole} and Corollary \ref{prop:Lp_Lq_estimates_bounded} we have
\begin{equation*}
\|W(t,s)f\|_{p,\Omega} \leq C\|f\|_{p,\Omega} \qquad \quad \mbox{for}\; f \in L^p_\sigma(\Omega), \; (t,s)\in \Lambda_T.
\end{equation*}
So, by (\ref{eq:norm_estimate_error}), we can apply Lemma \ref{lemma:iterated} with $R=W$, $S=\P_\Omega F,\,\alpha =0,\,\beta =-\frac{\gamma}{2}$ and $X_1=X_2=L^p_\sigma(\Omega)$. Thus, for any $f\in L^p_\sigma(\Omega)$,
the series $\sum_{k=0}^{\infty}T_k(t,s)f$ converges uniformly in $\Lambda_T$,
where $T_0(t,s)f=W(t,s)f$ and
\begin{equation}
T_{k+1}(t,s)f= \int_s^t T_k(t,r)\P_\Omega F(r,s)f \d r,\quad (t,s)\in \Lambda_T,\,f\in L^p_\sigma(\R^d).
\end{equation}
Since $T>0$ is arbitrary,
\begin{equation}\label{eq:representation_evolution_exterior}
T_{\Omega,b} (t,s):=\sum_{k=0}^{\infty}T_k(t,s),\quad(t,s)\in \Lambda
\end{equation}
is well-defined. It is easy to check that $T_{\Omega,b}(t,s)$ satisfies the integral equation \eqref{eq:integral_eq}. Moreover, from the strong continuity of $W(\cdot ,\cdot)$ and (\ref{eq:norm_estimate_error}) we deduce inductively that $T_k(\cdot ,\cdot)$ is strongly continuous
and hence, by the uniform convergence of the series we get the strong continuity of $T_{\Omega,b}(\cdot ,\cdot)$.

In order to show that $\{T_{\Omega ,b}(t,s)\}_{(t,s)\in \Lambda}$ leaves $Y_\Omega$ invariant, we proceed as in the proof of Theorem 3.1 in \cite{Hansel/Rhandi:2010} and consider the Banach space
$$Z :=\{f\in W_0^{1,p}(\Omega)^d\cap L^p_\sigma(\Omega) : |x|\cdot \nabla f_i(x) \in L^p(\Omega)^d\; \mbox{for} \; i=1,\ldots,d\}$$ endowed with the norm
$$\|f\|_{Z}:=\|f\|_{1,p,\Omega}+\||x|\cdot \nabla f\|_{p,\Omega},\quad f\in Z.$$
Proposition \ref{prop:evolution_whole}, Corollary \ref{prop:Lp_Lq_estimates_bounded} and (\ref{eq:norm_estimate_error}) permit us to apply
Lemma \ref{lemma:iterated} with $X_1=X_2=Z,\,R=W,\,S=\P_\Omega F,\alpha =0$ and $\beta =-\frac{\gamma}{2}$.
So, we obtain that $T_{\Omega ,b}(t,s)f \in Z$ for all $f\in Z$ and $(t,s)\in \Lambda$. Moreover, by taking
$X_1=W^{2,p}(\Omega)^d,\,X_2=W^{1,p}(\Omega)^d\cap L^p_\sigma(\Omega),\,R=W,\,S=\P_\Omega F,\,\alpha=-\frac{1}{2},\,\beta =-\frac{\gamma}{2}$ and applying Proposition \ref{prop:evolution_whole}, Corollary \ref{prop:Lp_Lq_estimates_bounded}
and (\ref{eq:norm_estimate_error}), it follows, by Lemma \ref{lemma:iterated}, that
  $T_{\Omega ,b}(t,s)f \in W^{2,p}(\Omega)$ for all $f\in W^{1,p}(\Omega)^d\cap L^p_\sigma(\Omega)$ and $(t,s)\in \widetilde{\Lambda}$. This yields that $\{T_{\Omega ,b}(t,s)\}_{(t,s)\in \Lambda}$ leaves $Y_\Omega$ invariant and
  \begin{align}\label{conv-in-Y_Omega}
  &\sum_{n=0}^\infty \left[\|T_k(t,s)f\|_{2,p,\Omega}+\||x|\cdot \nabla T_k(t,s)f\|_{p,\Omega}\right]\notag\\
  &\qquad\quad\quad <C_T(1+(t-s)^{-\frac{1}{2}})
   (\|f\|_{1,p,\Omega}+\||x|\cdot \nabla f\|_{p,\Omega}),\quad (t,s)\in \widetilde{\Lambda}_T ,\,f\in Y_\Omega.
  \end{align}

Let us now prove that for every $f\in Y_\Omega$ and for every $s\geq 0$ fixed, the map $t\mapsto T_{\Omega,b}(t,s)f$ is differentiable on $(s,\infty)$ and that \eqref{eq:derivative_evolution_exterior} holds. For $f\in Y_\Omega$ we compute
\begin{align*}
\partial_t T_0(t,s)f & = A_{\Omega,b}(t) T_0(t,s) f - \P_\Omega F(t,s)f\\
\partial_t T_1(t,s)f & = A_{\Omega,b}(t) T_1(t,s) f + \P_\Omega F(t,s)f - \int_s^t \P_\Omega F(t,r)\P_\Omega F(r,s)f \d r\\
\partial_t T_2(t,s)f & = A_{\Omega,b}(t) T_2(t,s) f +  \int_s^t \P_\Omega F(t,r)\P_\Omega F(r,s)f \d r \\
& \qquad - \int_s^t \int_{r_1}^t \P_\Omega F(t,r_2)\P_\Omega F(r_2,r_1)\P_\Omega F(r_1,s)f \d r_2 \d r_1.
\end{align*}
Inductively we see that
\begin{equation}\label{eq:derivative}
\partial_t \sum_{k=0}^n T_k(t,s)f  = A_{\Omega,b}(t) \sum_{k=0}^n T_k(t,s)f - R_n(t,s)f
\end{equation}
holds for $n\in \N$, where
\begin{equation*}
R_n(t,s)f := \int_s^t \int_{r_1}^t \ldots \int_{r_{n-1}}^t \P_\Omega F(t,r_n)\P_\Omega F(r_n,r_{n-1})\ldots \P_\Omega F(r_1,s)f\,\d r_n \ldots \d r_2 \d r_1.
\end{equation*}

We estimate now the norm of $R_n(t,s)f$. By Lemma \ref{lemma:error_terms} we obtain
\begin{align*}
\|R_1(t,s)f\|_{p,\Omega} &\leq C^2 \int_s^t (t-r)^{-\frac \gamma 2}(r -s)^{-\frac \gamma 2} \d r \|f\|_{p,\Omega}= C^2 \B(1-\gamma/2,1-\gamma/2)(t-s)^{1-\gamma}\|f\|_{p,\Omega},\\
\|R_2(t,s)f\|_{p,\Omega} & \leq C^3 \B(1-\gamma/2,1-\gamma/2)\int_s^t (t-r)^{1-\gamma}(r-s)^{-\frac{\gamma}{2}}\d r \|f\|_{p,\Omega}\\
& = C^3\B(1-\gamma/2,1-\gamma/2)\B(1-\gamma/2,2-\gamma)(t-s)^{2-\frac{3\gamma}{2}}\|f\|_{p,\Omega}.
\end{align*}
Inductively we see that
\begin{align}\label{eq:correction_terms}
\|R_n(t,s)f\|_{p,\Omega} &\leq C^{n+1}\B(1-\gamma/2,1-\gamma/2)\B(1-\gamma/2,2-\gamma)\ldots\notag\\
&\qquad\qquad  \ldots \B(1-\gamma/2, n-(n\gamma)/2)(t-s)^{n-\frac{(n+1)\gamma}{2}}\|f\|_{p,\Omega}\nonumber\\
& \leq \frac{C^{n+1}\Gamma(1-\gamma/2)^n}{\big[n-\frac{(n+1)\gamma}{2}\big]!}(t-s)^{n-\frac{(n+1)\gamma}{2}}\|f\|_{p,\Omega}
\end{align}
holds for $n\in \N$. Here the constant $C$ may change from line to line. From estimate (\ref{eq:correction_terms}) it follows that $\|R_n(t,s)\|_p$ tends to zero as $n\rightarrow \infty$. Since we used Lemma \ref{lemma:iterated}, we know that
 the convergence of $\sum_{k=0}^\infty T_k(t,s)f$ is uniform in $\Lambda_T$ for any $f\in Y_\Omega$ and any $T>0$ and so, by using
 (\ref{conv-in-Y_Omega}) and the closedness of $A_{\Omega,b}(t)$ we can conclude that
\begin{equation*}
\frac{\partial}{\partial t} \sum_{k=0}^{\infty} T_k(t,s)f  = A_{\Omega,b}(t) \sum_{k=0}^{\infty} T_k(t,s)f
\end{equation*}
 holds and this proves (\ref{eq:derivative_evolution_exterior}).

Let us show now the differentiability of the map $s\mapsto T_{\Omega,b}(t,s)f$ on $[0,t)$ for $t>0$ and $f\in Y_\Omega$. First we
note by a short calculation that for $f\in \D(L_{\Omega,b}(s))$
\begin{align*}
L_{\R^d}(s)f_R &= (L_{\Omega,b} (s)f)_R + 2 \nabla f \cdot (\nabla \xi)^* + (\Delta \xi + (M(s)x + c(s))\cdot \nabla \xi) f \\
&\qquad\quad - L_{\R^d}(s)\BB_1((\nabla\xi) \cdot f) + \BB_1((\nabla\xi)\cdot L_{\Omega,b}(s)f)
\end{align*}
 and
 \begin{align*}
 L_{D,b}(s)f_D & = (L_{\Omega,b}(s) f)_D + 2 \nabla f \cdot (\nabla \eta)^* + (\Delta \eta + (M(s)x + c(s))\cdot \nabla \eta) f \\
 &\qquad\quad - L_{D,b}(s)\BB_3((\nabla\eta) \cdot f) + \BB_3((\nabla\eta)\cdot L_{\Omega,b}(s)f).
 \end{align*}
 In the following we set $\nabla \tilde\p := (\Id-\P_\Omega)L_{\Omega,b}(s)f$. Then we can conclude that for any
 $f\in\D(A_{\Omega,b}(s))$ we have
 \begin{align*}
 \P_{\R^d}(L_{\Omega,b} (s)f)_R &= \P_{\R^d}\left[\xi A_{\Omega,b}(s)f - \BB_1((\nabla \xi)A_{\Omega,b}(s)f)+\xi\nabla\tilde \p - \BB_1((\nabla\xi)\nabla \tilde \p) \right]\\
 & = (A_{\Omega,b}f)_R - \P_{\R^d}[\tilde \p \nabla \xi  + \BB_1((\nabla \xi)\nabla \tilde \p)],
 \end{align*}
 and analogously
 \begin{align*}
 \P_{D}(L_{\Omega,b} (s)f)_D &= (A_{\Omega,b}f)_D - \P_{D}[\tilde \p \nabla \eta  + \BB_3((\nabla \eta)\nabla \tilde \p)].
 \end{align*}
 Thus, for $f\in Y_\Omega$, we obtain
 \begin{align*}
 \frac{\partial}{\partial s} W(t,s)f & = -\varphi T_{\R^d}(t,s)A_{\R^d}(s) f_R - (1-\varphi)T_{D,b}(t,s)A_{D,b}(s)f_D\\
 &\qquad \quad + \BB_2((\nabla \varphi)\cdot (T_{\R^d}(t,s)A_{\R^d}(s)f_R-T_{D,b}(t,s)A_{D,b}(s)f_D))\\[0.15cm]
 & = - W(t,s)A_{\Omega,b}(s)f - G(t,s)f
 \end{align*}
 where
 \begin{align*}
 G(t,s)f := &\; \varphi T_{\R^d}(t,s)\P_{\R^d}\left[2 \nabla f \cdot (\nabla \xi)^* + (\Delta \xi + (M(s)x + c(s))\cdot \nabla \xi) f \right.\\
&\left. \qquad\qquad - L_{\R^d}(s)\BB_1((\nabla\xi) \cdot f) + \BB_1((\nabla\xi)\cdot L_{\Omega,b}(s)f) - \tilde \p \nabla \xi  - \BB_1((\nabla \xi)\nabla \tilde \p)\right]\\
& \;+ (1-\varphi) T_{D,b}(t,s)\P_D\left[2\nabla f \cdot (\nabla \eta)^* + (\Delta \eta + (M(s)x + c(s))\cdot \nabla \eta)f \right.\\
 &\left. \qquad\qquad - L_{D,b}(s)\BB_3 ((\nabla\eta) \cdot f) + \BB_3((\nabla\eta)\cdot L_{\Omega,b}(s)f)-\tilde \p \nabla \eta  - \BB_3((\nabla \eta)\nabla \tilde \p)\right]\\
 &\; - \BB_2\left\{ (\nabla \varphi) \cdot T_{\R^d}(t,s)\P_{\R^d}\left[2 \nabla f \cdot (\nabla \xi)^* + (\Delta \xi + (M(s)x + c(s))\cdot \nabla \xi) f\right. \right. \\
&\left. \left. \qquad\qquad - L_{\R^d}(s)\BB_1((\nabla\xi) \cdot f) + \BB_1((\nabla\xi)\cdot L_{\Omega,b}(s)f) - \tilde \p \nabla \xi  - \BB_1((\nabla \xi)\nabla \tilde \p)\right]  \right\}\\
 & \; + \BB_2\left\{(\nabla \varphi) \cdot T_{D,b}(t,s)\P_D\left[2\nabla f \cdot (\nabla \eta)^* + (\Delta \eta + (M(s)x + c(s))\cdot \nabla \eta)f \right. \right. \\
 &\left. \left. \qquad\qquad - L_{D,b}(s)\BB_3 ((\nabla\eta) \cdot f) + \BB_3((\nabla\eta)\cdot L_{\Omega,b}(s)f) - \tilde \p \nabla \eta  - \BB_3((\nabla \eta)\nabla \tilde \p)\right]\right\}
 \end{align*}
 We estimate now the norm of $\tilde p$ in $L^p(D)$ since in the expression of $G(t,s)f$ the supports of the functions
 $\tilde p\nabla \xi$ and $\tilde p\nabla \eta$ are subsets of $D$.
Analogously to the proof of Lemma \ref{lem:pressure} we can show that for any $\varphi \in C_c^\infty(D)$
\begin{eqnarray*}
\left|\langle \tilde p ,\varphi \rangle_D\right| &\le &C
\|f \|_{\gamma,p,D}\|\varphi \|_{p',D}\\
&\le &C
\|f \|_{\gamma,p,\Omega}\|\varphi \|_{p',D}
\end{eqnarray*} and hence
$$\|\tilde p\|_{p,D}\le C\|f\|_{\gamma,p,\Omega}$$
holds, where $\gamma \in (1+\frac 1 p , 2)$. Now, using the same arguments as in the proof of Lemma \ref{lemma:error_terms} we obtain that
 $$
 \|G(t,s)f\|_{\gamma,p,\Omega}\leq C (t-s)^{-\frac {\gamma}{2}}\|f\|_{\gamma,p,\Omega}
 $$
 for some constant $C>0$ and $f\in H^{\gamma,p}(\Omega)^d\cap L^p_\sigma(\Omega)$. Now we apply Lemma \ref{lemma:iterated}  with $X_1=X_2=H^{\gamma,p}(\Omega)^d\cap L^p_\sigma(\Omega)$, $R=S=G$ and $\alpha=\beta=-\tfrac{\gamma}{2}$. So, the series
 \begin{equation*}
V(t,s)f := \sum_{k=0}^{\infty} V_k (t,s)f, \qquad (t,s)\in \widetilde{\Lambda},
\end{equation*}
is well-defined and
\begin{equation}\label{eq-22}
\|V(t,s)f\|_{\gamma,p,\Omega}\le C(t-s)^{-\frac{\gamma}{2}}\|f\|_{\gamma,p,\Omega},\quad (t,s)\in \widetilde{\Lambda}_T,
\end{equation}
 for $f\in H^{\gamma,p}(\Omega)^d\cap L^p_\sigma(\Omega)$. On the other hand, $V(\cdot,\cdot)$ satisfies the integral equation
\begin{equation}\label{eq:integral}
V(t,s)f = G(t,s)f + \int_s^t V(t,r)G(r,s)f \d r, \qquad (t,s)\in \widetilde{\Lambda},\,f\in H^{\gamma,p}(\Omega).
\end{equation}
In particular $V(t,\cdot)f$ is continuous on $[0,t)$ with respect to the $L^p$-norm for any $f\in H^{\gamma,p}(\Omega)^d\cap L^p_\sigma(\Omega)$ and
$t>0$. Now, for $f\in L^p_\sigma(\Omega)$ and $(t,s)\in \Lambda$ we set
\begin{equation*}
\tilde T(t,s)f := W(t,s)f+\int_s^t V(t,r)W(r,s)f \d r.
\end{equation*}
It follows from the continuity of $V(t,\cdot)W(\cdot ,s)f$ on $[s,t)$, Proposition \ref{prop:evolution_whole}, Proposition \ref{prop:evolution_bounded} and (\ref{eq-22}) that the above integral is well-defined for any $f\in L^p_\sigma(\Omega)$.
Computing the derivative with respect to $s\in [0,t)$ yields
\begin{align*}
\frac{\partial}{\partial s} \tilde T(t,s)f &= - W(t,s)L_\Omega(s)f - G(t,s)f+ V(t,s)f - \int_s^t V(t,r) W(r,s)L_\Omega(s) f \d r \\
& \qquad \quad - \int_s^t V(t,r)G(r,s)f \d r  \\[0.2cm]
& = - \tilde T(t,s)L_\Omega(s)f ,
\end{align*}
for any $f\in Y_\Omega$, due to (\ref{eq:integral}). From this equality together with (\ref{eq:derivative_evolution_exterior}) and since
$T_{\Omega,b}(t,s)Y_\Omega \subset Y_\Omega,\,(t,s)\in \Lambda$ we can conclude that
\begin{align*}
\frac{\partial}{\partial r} (\tilde T(t,r)T_{\Omega,b}(r,s)f)= 0
\end{align*}
holds for all $f\in Y_\Omega$. This yields that for $f\in Y_\Omega$, the function $\tilde T(t,r)T_{\Omega,b}(r,s)f$ is constant on $\Lambda_T $ and thus, by the density of $Y_\Omega$ in $L^p_\sigma(\Omega)$ and by the fact that $T>0$ was arbitrary, it follows that $\tilde T(t,s)f=T_{\Omega,b}(t,s)f$ holds for all $f\in L^p_\sigma(\Omega)$, $(t,s)\in \Lambda$. This proves  (\ref{eq:derivative_evolution_exterior_2}).

The uniqueness of the evolution system $\{T_{\Omega,b}(t,s)\}_{(t,s)\in \Lambda}$ follows by a similar argument from \eqref{eq:derivative_evolution_exterior} and \eqref{eq:derivative_evolution_exterior_2} (see the proof of Theorem 3.1 in \cite{Hansel/Rhandi:2010} for details).

The estimate in (d) can be obtained by applying Lemma \ref{lemma:iterated} with $X_1=W^{1,p}(\Omega),\,X_2=L^p_\sigma(\Omega),\,R=W,\,S=\P_\Omega F,\,\alpha =-\frac{1}{2},\,\beta=-\frac{\gamma}{2}$, Proposition \ref{prop:evolution_whole}  and Corollary  \ref{prop:Lp_Lq_estimates_bounded}.

Finally, the first estimate in (e) follows by applying Lemma \ref{lemma:iterated} with $X_1=L^q_\sigma(\Omega),\,X_2=L^p_\sigma(\Omega),\,R=W,\,S=\P_\Omega F,\,\alpha=
-\frac{d}{2}\left(\frac 1 p - \frac 1 q\right),\,\beta =-\frac{\gamma}{2}$,
Proposition \ref{prop:evolution_whole} and Corollary \ref{prop:Lp_Lq_estimates_bounded} if $\frac{1}{p}-\frac{1}{q}\le \frac{2}{d}$.
The case $\frac{1}{p}-\frac{1}{q}> \frac{2}{d}$ can be obtained by iteration as in the proof of Corollary \ref{prop:Lp_Lq_estimates_bounded}.
The second estimate follows now from the first one and (d) (see the proof of Corollary  \ref{prop:Lp_Lq_estimates_bounded}).
 \end{proof}
To conclude this section let us state the following lemma about the behavior of $T_{\Omega,b}(t,s)$ near $t=s$, which will be needed in the proof of Theorem \ref{thm:Kato}.
\begin{lemma}\label{lem5.3}
Let $s\geq0$ and $f\in L^p_{\sigma}(\Omega)$. Then we have
\begin{enumerate}
\item[(a)] for $1< p<q<\infty$
$$
	\lim\limits_{t\to s,\ t> s}(t-s)^{\frac d 2 \left(\frac 1 p - \frac 1 q\right)} \|T_{\Omega,b}(t,s)f\|_{q,\Omega}
	=0,$$
\item[(b)] for $1<p<\infty$
$$	\lim\limits_{t\to s,\ t> s}
(t-s)^{\frac 1 2 } \|\nabla T_{\Omega,b}(t,s)f\|_{p,\Omega}=0.$$
\end{enumerate}
\end{lemma}
\begin{proof}
Let $f\in L^p_\sigma(\Omega)$, $0\le t-s \leq 1$ and choose $(f_n)_{n\in\N}
\subset C_{c,\sigma}^{\infty}(\Omega) \subset L_{\sigma}^p(\Omega)$,
such that $\lim_{n\to\infty}\|f-f_n\|_{p,\Omega}=0$. The triangle inequality
together with the $L^p$-$L^q$ estimates stated in Theorem \ref{thm:evolution_system_exterior}
imply that there exist constants $C_1, C_2>0$ such that
\begin{align*}
&(t-s)^{\frac d 2 \left(\frac 1 p - \frac 1 q \right)}
\|T_{\Omega ,b}(t,s)f\|_{q,\Omega}\\
&\qquad \leq (t-s)^{\frac d 2 \left(\frac 1 p - \frac 1 q \right)}
\|T_{\Omega ,b}(t,s)(f-f_n)\|_{q,\Omega} + (t-s)^{\frac d 2 \left(\frac 1 p - \frac 1
q
\right)} \|T_{\Omega ,b}(t,s)f_n\|_{q,\Omega}\\
&\qquad \leq  C_1 \|f-f_n\|_{p,\Omega} + C_2 (t-s)^{\frac d 2 \left(\frac 1 p -
\frac 1 q \right)} \|f_n\|_{q,\Omega},\quad 0\leq t-s\leq 1,\
n\in\N.
\end{align*}
Hence,
$
	\lim\limits_{t\to s}(t-s)^{\frac d 2 \left(\frac 1 p - \frac 1 q
	\right)} \|T_{\Omega ,b}(t,s)f\|_{q,\Omega}=0
$
by letting first $t\rightarrow s$ and then $n\rightarrow \infty$.
The second assertion is proved in a similar
way (cf. \cite[Proposition 3.4]{Hansel:2009}).
\end{proof}
\section{Mild solutions to the nonlinear problem}
In this section we finally come back to the nonlinear problem \eqref{eq:NS_5} and its abstract formulation \eqref{eq:NS_abstract}. Based on the results proved in Section \ref{sect:exterior} and an adaptation of the Kato iteration procedure (see \cite{Kato:1984}) we now prove Theorem \ref{thm:Kato}.
\begin{proof}[Proof of Theorem  \ref{thm:Kato}.]
Let $d\le p<q$, $f\in L^p_\sigma(\Omega)$ and $T>0$. For $0<t\le T$ and $k\in \N$ we define
$u_1(t):=T_{\Omega ,b}(t,0)f$ and
$$u_{k+1}(t)=u_1(t)-\int_0^tT_{\Omega ,b}(t,s)\P_\Omega (u_k\cdot \nabla u_k)(s)\,\d s+
\int_0^tT_{\Omega ,b}(t,s)\P_\Omega F_1(s)\,\d s.$$
Let us prove that, for some $T_0>0$, the sequence $(u_k)_k$ converges in $C([0,T_0); L^p_\sigma(\Omega))$ to a mild solution $u$ of (\ref{eq:NS_abstract}) for $T\in (0,T_0)$.

We set $\beta :=\frac{d}{2}\left(\frac{1}{p}-\frac{1}{q}\right)$ and define
$$L_k:=L_k(T):=\sup_{t\in (0,T]}t^\beta \|u_k(t)\|_{q,\Omega},\,\,L'_k:=L'_k(T):=\sup_{t\in (0,T]}t^{\frac{1}{2}}\|\nabla u_k\|_{p,\Omega}$$ and
$$M_k:=M_k(T):=\sup_{t\in (0,T]}t^\beta \|u_{k+1}(t)-u_k(t)\|_{q,\Omega},\,M'_k:=M'_k(T):=
\sup_{t\in (0,T]}t^{\frac{1}{2}}\|\nabla u_{k+1}(t)-\nabla u_k(t)\|_{p,\Omega}$$ for $k\in \N$.

It follows from the $L^p$-$L^q$ estimates (\ref{Lpq}) and the boundedness of $\P_\Omega$ from $L^r(\Omega)^d$ into $L^r_\sigma(\Omega)$ that
\begin{eqnarray*}
\|u_{k+1}(t)\|_{q,\Omega} &\le & \|u_1(t)\|_{q,\Omega}+\int_0^t\left\|T_{\Omega ,b}(t,s)\P_\Omega (u_k\cdot \nabla u_k)(s)\right\|_{q,\Omega}\,\d s\\
& & \quad +\int_0^t\left\|T_{\Omega ,b}(t,s)\P_\Omega F_1(s)\right\|_{q,\Omega}\,\d s\\
&\le & t^{-\beta}L_1+C\int_0^t(t-s)^{-\frac{d}{2}\left(\frac{1}{r}-\frac{1}{q}\right)}\|u_k(s)\cdot \nabla u_k(s)\|_{r,\Omega}\,\d s\\
& & \quad +C\int_0^t\|F_1(s)\|_{q,\Omega}\,\d s\\
&\le & t^{-\beta}L_1+C\int_0^t(t-s)^{-\frac{d}{2}\left(\frac{1}{r}-\frac{1}{q}\right)}\|u_k(s)\|_{q,\Omega}\|\nabla u_k(s)\|_{p,\Omega}\,\d s\\
& & \quad +C\int_0^t\|F_1(s)\|_{q,\Omega}\,\d s\\
&\le & t^{-\beta}L_1+Ct^{-\beta}L_kL'_k\int_0^t(t-s)^{-\frac{d}{2p}}t^{\beta}s^{-\beta -\frac{1}{2}}\,\d s+
Ct^{-\beta}\int_0^t t^{\beta}\|F_1(s)\|_{q,\Omega}\,\d s,
\end{eqnarray*}
where $\frac{1}{r}=\frac{1}{p}+\frac{1}{q}$.
Since $\beta <\frac{1}{2}$ we obtain, by multiplying with $t^\beta$,
$$t^\beta \|u_{k+1}(t)\|_{q,\Omega} \le  L_1+C_1L_kL'_k+C_2T,$$
and taking the supremum over $t\in (0,T]$ yields
\begin{equation}\label{suite L}
L_{k+1}\le L_1+C_1L_kL'_k+C_2T,
\end{equation}
where $C_1,\,C_2>0$ are constants independent of $k\in \N$ but depend on $T$. Similarly, with the gradient $L^p$-$L^q$ estimates (\ref{nabla-Lpq}) we have
\begin{eqnarray}\label{eq6.2}
\|\nabla u_{k+1}(t)\|_{p,\Omega} &\le & t^{-\frac{1}{2}}L'_1+C\int_0^t(t-s)^{-\frac{d}{2}\left(\frac{1}{r}-\frac{1}{p}\right)-\frac{1}{2}}\|u_k(s)\cdot \nabla u_k(s)\|_{r,\Omega}\,\d s \nonumber \\
& & \quad +C\int_0^t(t-s)^{-\frac{1}{2}}\|F_1(s)\|_{p,\Omega}\,\d s \nonumber \\
&\le & t^{-\frac{1}{2}}L'_1+C\int_0^t(t-s)^{-\frac{d}{2q}-\frac{1}{2}}\|u_k(s)\|_{q,\Omega}\|\nabla u_k(s)\|_{p,\Omega}\,\d s\\
& & \quad +C\int_0^t(t-s)^{-\frac{1}{2}}\|F_1(s)\|_{p,\Omega}\,\d s \nonumber
\end{eqnarray}
and hence, as above, we obtain
\begin{equation}\label{suite L'}
L'_{k+1}\le L'_1+C_3L_kL'_k+C_4T,
\end{equation}
where $C_3,\,C_4>0$ are constants independent of $k\in \N$ but depend on $T$.
By setting $R_k:=R_k(T):=\max\{L_k,L'_k\}$, it follows from (\ref{suite L}) and (\ref{suite L'}) that $R_{k+1}\le R_1+c_1R_k^2+c_2T$ for some constants $c_1,\,c_2\ge 1$. By Lemma \ref{lem5.3} for any $\varepsilon >0$, there is $\widetilde{T}_0>0$ such that $R_1<\varepsilon$ for $T<\widetilde{T}_0$.
So, for $\varepsilon \le \frac{1}{8c_1}$, it follows by induction that $$R_k\le 4\varepsilon ,\quad k\in \N,\,T<\min\left(\widetilde{T}_0,\frac{\varepsilon}{c_2}\right)=:T_0.$$ Thus, the sequences
\begin{equation}\label{suites}
(t\mapsto t^\beta u_k(t))_k\quad \hbox{\ and}\quad (t\mapsto t^{\frac{1}{2}}\nabla u_k(t))_k
\end{equation}
 are uniformly bounded in
$L^q_\sigma(\Omega)$ and $L^p(\Omega)^{d\times d}$ respectively for $t\le T$ and all $k\in \N$. The continuity of the maps
$t\mapsto t^\beta u_1(t)$ and $t\mapsto t^{\frac{1}{2}}\nabla u_1(t)$ at $t=0$ follows from Lemma \ref{lem5.3}. Hence the continuity
of $t\mapsto t^\beta u_k(t)$ and $t\mapsto t^{\frac{1}{2}}\nabla u_k(t)$ follows from similar calculations.

We prove now that the sequences in (\ref{suites}) are Cauchy sequences. Since
$$u_k\cdot \nabla u_k-u_{k-1}\cdot \nabla u_{k-1}=u_k\cdot \nabla (u_k-u_{k-1})+(u_k-u_{k-1})\cdot \nabla u_{k-1}$$ and by the same computations as above we obtain
\begin{eqnarray*}
\|u_{k+1}(t)-u_k(t)\|_{q,\Omega} &\le & \int_0^t \left\|T_{\Omega ,b}(t,s)\P_\Omega\left((u_k\cdot \nabla u_k)(s)-(u_{k-1}\cdot \nabla u_{k-1})(s)\right)\right\|_{q,\Omega}\,\d s\\
&\le & C\int_0^t (t-s)^{-\frac{d}{2p}}\left(\|u_k(s)\|_{q,\Omega}\|\nabla (u_k(s)-u_{k-1}(s))\|_{p,\Omega}\right. \\
& & \quad \quad \left. +\|u_k(s)-u_{k-1}(s)\|_{q,\Omega}\|\nabla u_{k-1}(s)\|_{p,\Omega}\right)\,\d s
\end{eqnarray*}
and
\begin{eqnarray*}
\|\nabla u_{k+1}(t)-\nabla u_k(t)\|_{p,\Omega} &\le & \int_0^t \left\|\nabla T_{\Omega ,b}(t,s)\P_\Omega\left((u_k\cdot \nabla u_k)(s)-(u_{k-1}\cdot \nabla u_{k-1})(s)\right)\right\|_{p,\Omega}\,\d s\\
&\le & C\int_0^t (t-s)^{-\frac{d}{2q}-\frac{1}{2}}\left(\|u_k(s)\|_{q,\Omega}\|\nabla (u_k(s)-u_{k-1}(s))\|_{p,\Omega}\right. \\
& & \quad \quad \left. +\|u_k(s)-u_{k-1}(s)\|_{q,\Omega}\|\nabla u_{k-1}(s)\|_{p,\Omega}\right)\,\d s.
\end{eqnarray*}
Therefore,
\begin{eqnarray*}
M_k &\le & C_5(M'_{k-1}L_k+M_{k-1}L'_{k-1})\le 2C_5R_1(M'_{k-1}+M_{k-1}),\\
M'_k &\le & C_6(M'_{k-1}L_k+M_{k-1}L'_{k-1})\le 2C_6R_1(M'_{k-1}+M_{k-1}),
\end{eqnarray*}
and hence,
\begin{eqnarray*}
(M_k+M'_k) &\le & 4\varepsilon (C_5+C_6)(M_{k-1}+M'_{k-1})\\
&\le & \frac{1}{2}(M_{k-1}+M'_{k-1})
\end{eqnarray*}
for $\varepsilon \le \min\left(\frac{1}{8(C_5+C_6)},\frac{1}{8c_1}\right)$ and $T<T_0$,
where $C_5,\,C_6>0$ are constants independent of $k\in \N$ but depend on $T$.
Thus,
$(t\mapsto t^\beta u_k(t))_k$ converges to some $t\mapsto t^\beta u(t)\in C([0,T], L^q_\sigma(\Omega))$ and
$(t\mapsto t^{\frac{1}{2}}\nabla u_k(t))_k$ converges to some $t\mapsto t^{\frac{1}{2}}v\in C([0,T], L^p(\Omega)^{d\times d})$. It is now clear that $v(t)=\nabla u(t)$ and $u$ is a mild solution of \eqref{eq:NS_abstract} on $[0,T]$.

The case $d<p=q$ follows by taking $\beta =\frac{1-\delta}{2}$ for $\delta \in (0,1)$ in the above computations. Here we need
$\frac{d}{2q}+\frac{1}{2}<1$ to make the integral in (\ref{eq6.2}) convergent.

Moreover, take any $p\ge d$ and any $r>p$ and using the same estimates as for $\|u_{k+1}(t)\|_{q,\Omega}$ above, we obtain
\begin{eqnarray*}
\|u(t)\|_{p,\Omega} &\le & \|T_{\Omega ,b}(t,0)f\|_{p,\Omega}+\int_0^t \left\|T_{\Omega ,b}(t,s)\P_\Omega (u\cdot \nabla u)(s)\right\|_{p,\Omega}\,\d s+CT\\
&\le & C\|f\|_{p,\Omega}+C(\sup_{t\in [0,T]}t^\beta \|u(t)\|_{r,\Omega})(\sup_{t\in [0,T]}t^{\frac{1}{2}} \|\nabla u(t)\|_{p,\Omega})\int_0^t (t-s)^{-\frac{d}{2r}}s^{-\beta -\frac{1}{2}}\,\d s\\
& & \quad \quad \quad \quad +CT
\end{eqnarray*}
and hence $\sup_{t\in [0,T]} \|u(t)\|_{p,\Omega}<\infty$. The continuity can be obtained similarly. Thus, $u\in C([0,T],L^p_\sigma(\Omega))$.

The case $d=p=q$ can be deduced as in the case $d<p=q$ by taking $\beta=\frac{1-\delta}{2}$ for $\delta \in (0,1)$ and making the some computations but with a $q'>d$ instead of $q=d$, since if $q=d$ then the integral in (\ref{eq6.2}) could diverge.

The property (\ref{eq:mild_solution_prop3}) follows from the construction of the solution $u$ and
Lemma \ref{lem5.3}.
Finally the uniqueness follows as in \cite{Hansel:2009}.
\end{proof}
\section*{Acknowledgments}

Parts of the paper were written while the first author stayed at Waseda University in Tokyo. He would like to express his thank to Professor Y. Shibata for the kind hospitality and for fruitful discussions.

The first author acknowledges the financial support of the DFG International Research Training Group 1529 on \emph{Mathematical Fluid
Dynamics} at TU Darmstadt and of the University of Salerno where the paper originated.

The second author acknowledges the financial support of the M.I.U.R. research project
Prin 2008 ``Metodi deterministici e stocastici nello studio di problemi di evoluzione''.
%
\bibliographystyle{alpha}

%
%
%
%
\end{document}